\documentclass[11pt,leqno]{amsart}

\usepackage[margin=1.5in]{geometry}
\usepackage[colorinlistoftodos]{todonotes}

\newcounter{todocounter}

\usepackage[T1]{fontenc}
\usepackage{tgpagella}
\usepackage{mathpazo}

\usepackage[pagebackref]{hyperref} 
\usepackage{xcolor}
\usepackage{amsmath,amsthm}
\usepackage{amssymb}
\usepackage{graphicx}
\usepackage[mathscr]{eucal}
\usepackage{MnSymbol}
\usepackage[frame,ps,matrix,arrow,curve,rotate,all,2cell,tips]{xy}
\usepackage{epic,eepic}
\usepackage{enumerate}
\usepackage{tikz-cd}

\newtheorem{main}{Theorem}

\newtheorem{theorem}[equation]{Theorem}
\newtheorem{lemma}[equation]{Lemma}

\newtheorem{proposition}[equation]{Proposition}
\newtheorem{corollary}[equation]{Corollary}

\theoremstyle{definition}
\newtheorem{definition}[equation]{Definition}
\newtheorem{example}[equation]{Example}
\newtheorem{remark}[equation]{Remark}

\newtheorem{conjecture}[equation]{Conjecture}
\newtheorem{problem}[equation]{Problem}

\numberwithin{equation}{subsection}

\newcommand{\F}{\mathbb{F}}

\renewcommand{\O}{\mathcal{O}}
\newcommand{\RO}{R_{\O}}
\newcommand{\arrowo}{\overrightarrow{\O}}
\newcommand{\arrowotensor}{\arrowo^{\otimes}}
\newcommand{\arrowopp}{\arrowo^{\square}}
\newcommand{\arrowcom}{\overrightarrow{\Com}}
\newcommand{\arrowcomtensor}{\arrowcom^{\otimes}}
\newcommand{\arrowcompp}{\arrowcom^{\square}}

\newcommand{\algom}{\alg(\O;\M)}
\newcommand{\algomar}{\overrightarrow{\algom}}
\newcommand{\algominv}{\alg(\O;\M[W^{-1}])}
\newcommand{\algnom}{\alg(N^{\otimes}\O;N(\M^c)[W^{-1}])}
\newcommand{\algosm}{\alg(\O^s;\M)}

\newcommand{\algomtensor}{\alg\bigl(\arrowotensor; \arrowmtensor\bigr)}
\newcommand{\algompp}{\alg\bigl(\arrowopp; \arrowmpp\bigr)}

\newcommand{\arrowastensor}{\overrightarrow{\As}^{\otimes}}

\newcommand{\M}{\mathsf{M}}

\newcommand{\Msigmanop}{\M^{\sigmaop_n}}

\newcommand{\Sigmanop}{\Sigmaop_n}

\newcommand{\N}{\mathsf{N}}

\newcommand{\C}{\mathcal{C}}

\newcommand{\D}{\mathcal{D}}

\newcommand{\sD}{\mathscr{D}}

\renewcommand{\emptyset}{\varnothing}

\DeclareMathOperator{\dom}{dom}

\DeclareMathOperator{\colim}{colim}
\newcommand{\colimover}[1]{\underset{#1}{\colim}}
\DeclareMathOperator{\coker}{coker}
\newcommand{\Cok}{\mathsf{C}}
\newcommand{\Ker}{\mathsf{K}}

\newcommand{\po}{\ar@{}[dr]|(.7){\Searrow}}
\newcommand{\pb}{\ar@{}[dr]|(.3){\Nwarrow}}

\newcommand{\Ch}{\mathsf{Ch}}

\newcommand{\cat}[1]{\mathcal{#1}}

\newcommand{\boxprod}{\mathbin\square}
\newcommand{\lbox}{\largesquare}
\newcommand{\boxprodover}[1]{\underset{#1}{\boxprod}}
\newcommand{\pcorner}{\circledast}
\newcommand{\pback}{\boxbackslash} 

\newcommand{\Sp}{\mathsf{Sp}}

\newcommand{\coprodover}[1]{\underset{#1}{\coprod}}

\newcommand{\tensorover}[1]{\underset{#1}{\otimes}}

\newcommand{\bigtensor}[2]{\overset{#2}{\underset{#1}{\bigotimes}}}

\newcommand{\timesover}[1]{\underset{#1}{\times}}
\newcommand{\algo}{{\alg(\O)}}

\newcommand{\smallbinom}[2]
{\raisebox{.05cm}{\scalebox{0.8}{$\binom{#1}{#2}$}}}

\usepackage{tikz}
\usetikzlibrary{arrows,decorations.pathmorphing}
\usetikzlibrary{backgrounds,positioning}
\usetikzlibrary{fit,petri,shapes.misc}

\tikzset{auto}

\tikzset{empty/.style={circle,inner sep=0pt,minimum size=6mm}}
\tikzset{emptyvt/.style={circle,inner sep=0pt,minimum size=0mm}}

\tikzset{plain/.style={circle,draw,very thick,
inner sep=0pt,minimum size=6mm}}

\tikzset{fatplain/.style={rounded rectangle,draw,very thick,minimum size=6mm}}

\tikzset{bigplain/.style={rounded rectangle,draw,very thick,minimum size=.8cm}}

\tikzset{yellowvt/.style={circle,draw,fill=yellow,very thick,inner sep=0pt,minimum size=6mm}}

\tikzset{bluevt/.style={circle,draw,fill=blue!20,very thick,inner sep=0pt,minimum size=6mm}}

\tikzset{greenvt/.style={circle,draw,fill=green!30,very thick,inner sep=0pt,minimum size=6mm}}

\tikzset{redvt/.style={circle,draw,fill=red!30,very thick,inner sep=0pt,minimum size=6mm}}

\tikzset{arrow/.style={->,thick}}
\tikzset{dashedarrow/.style={->,dashed,thick}}
\tikzset{dottedarrow/.style={->,dotted,thick}}
\tikzset{mapto/.style={|->,thick}}

\tikzset{implies/.style={thick,double,double equal sign distance,-implies}}

\tikzset{line/.style={thick}}
\tikzset{dottedline/.style={dotted,thick}}
\tikzset{dashedline/.style={dashed,thick}}

\tikzset{inputleg/.style={<-,thick}}
\tikzset{outputleg/.style={->,thick}}
\tikzset{dottedinput/.style={<-,dotted,thick}}

\newcommand{\adjoint}{\hspace{-.1cm}
\nicearrow\xymatrix{ \ar@<2pt>[r] & \ar@<2pt>[l]}\hspace{-.1cm}}
\renewcommand{\hookrightarrow}{\nicexy{\ar@{^{(}->}[r] &}}
\newcommand{\nicearrow}{\SelectTips{cm}{10}}
\newcommand{\nicexy}{\nicearrow\xymatrix@C+10pt@R+8pt}
\newcommand{\narrowxy}{\nicearrow\xymatrix@R+10pt}

\renewcommand{\to}{\hspace{-.1cm}\nicearrow\xymatrix@C-.3cm{\ar[r]&}\hspace{-.1cm}}

\newcommand{\tensor}{\otimes}
\newcommand{\tensorunit}{\mathbb{1}}

\newcommand{\fC}{\mathfrak{C}}

\renewcommand{\sD}{\mathsf{D}}

\newcommand{\As}{\mathsf{As}}
\newcommand{\Com}{\mathsf{Com}}

\newcommand{\labar}{\overline{\lambda}}

\newcommand{\ua}{\underline{a}}
\newcommand{\ub}{\underline{b}}
\newcommand{\uc}{\underline{c}}

\newcommand{\smallop}{{\scalebox{.5}{$\mathrm{op}$}}}
\newcommand{\cof}{{\scalebox{.5}{$\mathrm{cof}$}}}

\newcommand{\tcof}{{\scalebox{.5}{$\mathrm{t.cof}$}}}
\newcommand{\clubcof}{(\clubsuit)_{\cof}}

\newcommand{\clubtcof}{(\clubsuit)_{\tcof}}

\newcommand{\alg}{\mathsf{Alg}}

\newcommand{\CAlg}{\mathsf{CAlg}}

\newcommand{\spancat}{\{\nicexy@C-.8cm{-1 & 0 \ar[l] \ar[r] & 1}\}}

\newcommand{\inj}{{\scalebox{.5}{$\mathrm{inj}$}}}
\newcommand{\proj}{{\scalebox{.5}{$\mathrm{proj}$}}}

\newcommand{\arrowm}{\overrightarrow{\M}}
\newcommand{\arrowminj}{\arrowm_{\inj}}
\newcommand{\arrowmproj}{\arrowm_{\proj}}
\newcommand{\arrowmtensor}{\arrowm^{\otimes}}
\newcommand{\arrowmtensorc}{(\arrowmtensor)^{\fC}}
\newcommand{\arrowmtensorinj}{\arrowmtensor_{\inj}}
\newcommand{\arrowmtensorsigmanop}{(\arrowmtensor)^{\Sigmaop_n}}

\newcommand{\arrowmpp}{\arrowm^{\square}}
\newcommand{\arrowmppc}{(\arrowmpp)^{\fC}}
\newcommand{\arrowmppproj}{\arrowmpp_{\proj}}

\newcommand{\Sigmac}{\Sigma_{\fC}}
\newcommand{\Sigmaop}{\Sigma^{\smallop}}
\newcommand{\sigmaop}{\Sigmaop}
\newcommand{\Sigmacop}{\Sigmac^{\smallop}}
\newcommand{\Sigmacopc}{\Sigmacop \times \fC}

\newcommand{\smallprof}[1]
{\raisebox{.05cm}{\scalebox{0.8}{#1}}}

\newcommand{\ciubi}
{\smallprof{$\binom{c_i}{\ub_i}$}}

\newcommand{\ccsingle}
{\smallprof{$\binom{c}{c}$}}
\newcommand{\ddsingle}
{\smallprof{$\binom{d}{d}$}}

\newcommand{\dub}
{\smallprof{$\binom{d}{\ub}$}}
\newcommand{\duc}
{\smallprof{$\binom{d}{\uc}$}}
\newcommand{\dnc}
{\smallprof{$\binom{d}{nc}$}}

\newcommand{\dnothing}
{\smallprof{$\binom{d}{\varnothing}$}}

\newcommand{\andspace}{\qquad\text{and}\qquad}

\renewcommand{\lim}{\mathsf{lim}\,}

\DeclareMathOperator{\Ev}{Ev}
\DeclareMathOperator{\Hom}{Hom}

\DeclareMathOperator{\Id}{Id}

\begin{document}

\title{Smith Ideals of Operadic Algebras in Monoidal Model Categories}

\author{David White}
\address{Denison University
\\ Granville, OH}
\email{david.white@denison.edu}

\author{Donald Yau}
\address{The Ohio State University at Newark \\ Newark, OH}
\email{yau.22@osu.edu}

\begin{abstract}
Building upon Hovey's work on Smith ideals for monoids, we develop a homotopy theory of Smith ideals for general operads in a symmetric monoidal category.  For a sufficiently nice stable monoidal model category and an operad satisfying a cofibrancy condition, we show that there is a Quillen equivalence between a model structure on Smith ideals and a model structure on algebra morphisms induced by the cokernel and the kernel.  For symmetric spectra, this applies to the commutative operad and all Sigma-cofibrant operads.  For chain complexes over a field of characteristic zero and the stable module category, this Quillen equivalence holds for all operads.  This paper ends with a comparison between the semi-model category  approach and the $\infty$-category approach to encoding the homotopy theory of algebras over Sigma-cofibrant operads that are not necessarily admissible.
\end{abstract}

\maketitle


\section{Introduction}

A major part of stable homotopy theory is the study of structured ring spectra.  These include strict ring spectra, commutative ring spectra, $A_{\infty}$-ring spectra, $E_{\infty}$-ring spectra, $E_n$-ring spectra, and so forth.  Based on an unpublished talk by Jeff Smith, in \cite{hovey-smith} Hovey developed a homotopy theory of Smith ideals for ring spectra and monoids in more general symmetric monoidal model categories.  

Let us briefly recall Hovey's work in \cite{hovey-smith}.  For a symmetric monoidal closed category $\M$, its arrow category $\arrowm$ is the category whose objects are morphisms in $\M$ and whose morphisms are commutative squares in $\M$.  It has two symmetric monoidal closed structures, namely, the tensor product monoidal structure $\arrowmtensor$ and the pushout product monoidal structure $\arrowmpp$.  A monoid in $\arrowmpp$ is a Smith ideal, and a monoid in $\arrowmtensor$ is a monoid morphism.  If $\M$ is a model category, then $\arrowmtensor$ has the injective model structure $\arrowmtensorinj$, where weak equivalences and cofibrations are defined entrywise, and the category of monoid morphisms inherits a model structure from $\arrowmtensorinj$.  Likewise, $\arrowmpp$ has the projective model structure $\arrowmppproj$, where weak equivalences and fibrations are defined entrywise, and the category of Smith ideals inherits a model structure from $\arrowmppproj$.  Surprisingly, when $\M$ is pointed (resp., stable), the cokernel and the kernel form a Quillen adjunction (resp., Quillen equivalence) between $\arrowmpp$ and $\arrowmtensor$ and also between Smith ideals and monoid morphisms.

Since monoids are algebras over the associative operad, a natural question is whether there is a satisfactory theory of Smith ideals for algebras over other operads.  For the commutative operad, the first author showed in \cite{white-commutative} that commutative Smith ideals in symmetric spectra, equipped with either the positive flat (stable) or the positive (stable) model structure, inherit a model structure.  The purpose of this paper is to generalize Hovey's work to Smith ideals for general operads in monoidal model categories.  For an operad $\O$ we define a Smith $\O$-ideal as an algebra over an associated operad $\arrowopp$ in the arrow category $\arrowmpp$.  We will prove a precise version of the following result in Theorem \ref{smith=map}.

\begin{main}
Suppose $\M$ is a sufficiently nice stable monoidal model category, and $\O$ is a $\fC$-colored operad in $\M$ such that cofibrant Smith $\O$-ideals are also entrywise cofibrant in the arrow category of $\M$ with the projective model structure.   Then there is a Quillen equivalence
\[\nicexy{\bigl\{\text{Smith $\O$-Ideals}\bigr\} \ar@<2pt>[r]^-{\coker} 
& \bigl\{\text{$\O$-Algebra Maps}\bigr\} \ar@<2pt>[l]^-{\ker}}\]
induced by the cokernel and the kernel.
\end{main}

For example, this Theorem holds in the following situations:
\begin{enumerate}
\item $\O$ is an arbitrary $\fC$-colored operad, and $\M$ is (i) the category $\Ch(R)$ of bounded or unbounded chain complexes over a semi-simple ring containing $\mathbb{Q}$ (Corollary \ref{chain-zero}), (ii) the stable module category of $k[G]$-modules for some field $k$ and finite group $G$ (Corollary \ref{stmod-alloperad}), or (iii) the category of classical, equivariant, or motivic symmetric spectra with the positive or positive flat stable model structure (Example \ref{stable-examples}).
\item $\O$ is the commutative operad, and $\M$ is any of the examples above or equivariant orthogonal spectra, Hausmann's $G$-symmetric spectra \cite{hausmann}, or Schwede's global equivariant spectra \cite{schwede-global} with positive flat model structures (Subsection \ref{subsec:comm}).
\item $\O$ is $\Sigmac$-cofibrant (e.g., the associative operad, $A_{\infty}$-operads, $E_{\infty}$-operads, and $E_n$-operads), and $\M$ is any of the examples above, or (i) $\Ch(R)$ for a commutative ring $R$, (ii) StMod$(k[G])$ where $k$ is a principle ideal domain, (iii) an injective or projective model structure on spectra, (iv) $S$-modules \cite{ekmm}, (v) Mandell's equivariant symmetric spectra \cite{mandell-equivariant}, or (vi) a Lydakis style model structure on enriched functors (Corollary \ref{sigmacof-smith=map}, Example \ref{applicable-operads}, and Example \ref{stable-examples2}).
\end{enumerate}

The rest of this paper is organized as follows.  In Section \ref{sec:model-arrow} we recall some basic facts about model categories and arrow categories.  In Section \ref{sec:smith-operad} we define Smith ideals for an operad and prove that, when $\M$ is pointed, there is an adjunction between Smith $\O$-ideals and $\O$-algebra morphisms given by the cokernel and the kernel.  In Section \ref{sec:homotopy-smith} we define the model structures on Smith $\O$-ideals and $\O$-algebra morphisms and prove the Theorem above. We also include a discussion of what happens when there are only semi-model structures on Smith $\O$-ideals and $\O$-algebra morphisms. In Section \ref{sec:com} we apply the Theorem to the commutative operad and $\Sigmac$-cofibrant operads.  In Section \ref{sec:entywise} we apply the Theorem to entrywise cofibrant operads. In Section \ref{sec:appendix} we include a comparison between various approaches to encoding the homotopy theory of operad-algebras, including model categories, semi-model categories, and $\infty$-categories. This discussion holds in general, beyond the situation of Smith $\O$-ideals and $\O$-algebra morphisms.

\subsection*{Acknowledgments}

The authors would like to thank Mark Hovey for laying the groundwork for the study of Smith ideals, for suggesting to the first author to figure out the homotopy theory of commutative Smith ideals, and for all the guidance he has given to the community on matters related to model categories. Furthermore, we thank Bob Bruner, Dan Isaksen, and Andrew Salch for encouraging us to think about Smith ideals operadically, we thank Adeel Khan, Tyler Lawson, and Denis Nardin for an email exchange about this project, and we thank Rune Haugseng for an extremely helpful discussion related to Section \ref{sec:appendix}, and for encouraging us to write this section. Lastly, we thank the referee for many helpful comments that improved the exposition.

\section{Model Structures on the Arrow Category}
\label{sec:model-arrow}

In this section we recall a few facts about monoidal model categories and arrow categories.  Our main references for model categories are \cite{hirschhorn,hovey,ss}.  In this paper, $(\M, \otimes, \tensorunit, \Hom)$  will usually be a bicomplete symmetric monoidal closed category \cite{maclane} (VII.7) with monoidal unit $\tensorunit$, internal hom $\Hom$, initial object $\varnothing$, and terminal object $*$. Since $\M$ is closed, $\varnothing \otimes X = \varnothing$ for any $X$.

\subsection{Monoidal Model Categories}

A model category is \emph{cofibrantly generated} if there are a set $I$ of cofibrations and a set $J$ of trivial cofibrations (that is, morphisms that are both cofibrations and weak equivalences) that permit the small object argument (with respect to some cardinal $\kappa$), and a morphism is a (trivial) fibration if and only if it satisfies the right lifting property with respect to all morphisms in $J$ (resp. $I$).

Let $I$-cell denote the class of transfinite compositions of pushouts of morphisms in $I$, and let $I$-cof denote retracts of such \cite{hovey} (2.1.9). In order to run the small object argument, we will assume the domains $K$ of the morphisms in $I$ (and $J$) are $\kappa$-small relative to $I$-cell (resp. $J$-cell).  In other words, given a regular cardinal $\lambda \geq \kappa$ and any $\lambda$-sequence $X_0\to X_1\to \cdots$ formed of morphisms $X_\beta \to X_{\beta+1}$ in $I$-cell, the map of sets
\[\nicexy{\colim_{\beta < \lambda} \M\bigl(K,X_\beta\bigr) \ar[r] 
& \M\bigl(K,\colim_{\beta < \lambda} X_\beta\bigr)}\]
is a bijection. An object is \emph{small} if there is some $\kappa$ for which it is $\kappa$-small. We will say that a model category is \emph{strongly cofibrantly generated} if the domains and codomains of $I$ and $J$ are small with respect to the entire category.


In Section \ref{sec:homotopy-smith}, we will produce homotopy theories for operad-algebras valued in arrow categories equipped with some model structure. Depending on the colored operad and properties of $\M$, sometimes we will only have a semi-model structure on a category of algebras.  However, as shown in Section \ref{sec:appendix}, it still encodes the correct $\infty$-category.  A semi-model category satisfies axioms similar to those of a model category, but one only knows that morphisms {\em with cofibrant domain} admit a factorization into a trivial cofibration followed by a fibration, and one only knows that trivial cofibrations {\em with cofibrant domain} lift against fibrations. To the authors' knowledge, every result about model categories has a corresponding result for semi-model categories, often obtained by first cofibrantly replacing everything in sight (see, for example, \cite{bous-loc-semi}). The following is Definition 2.1 in \cite{bous-loc-semi}.

\begin{definition} \label{defn:semi}
A \textit{semi-model structure} on a category $\M$ consists of classes of weak equivalences $W$, fibrations $F$, and cofibrations $Q$ satisfying the following axioms:

\begin{enumerate}
\item[M1] Fibrations are closed under pullback.
\item[M2] The class $W$ is closed under the two-out-of-three property.
\item[M3] $W,F,Q$ are all closed under retracts.
\item[M4] 
\begin{enumerate}
\item[i] Cofibrations have the left lifting property with respect to trivial fibrations.
\item[ii] Trivial cofibrations whose domain is cofibrant have the left lifting property with respect to fibrations.
\end{enumerate}
\item[M5] 
\begin{enumerate}
\item[i] Every morphism in $\M$ can be functorially factored into a cofibration followed by a trivial fibration. 
\item[ii] Every morphism whose domain is cofibrant can be functorially factored into a trivial cofibration followed by a fibration.
\end{enumerate} 
\end{enumerate}

If, in addition, $\M$ is bicomplete, then we call $\M$ a \textit{semi-model category}.
$\M$ is said to be \textit{cofibrantly generated} if there are sets of morphisms $I$ and $J$ in $\M$ such that the class of (trivial) fibrations is characterized by the right lifting property with respect to $J$ (resp. $I$), the domains of $I$ are small relative to $I$-cell, and the domains of $J$ are small relative to morphisms in $J$-cell whose domain is cofibrant. 
\end{definition}

An adjunction with left adjoint $L$ and right adjoint $R$ is denoted by $L \dashv R$.

\begin{definition}\label{quillen.pair}
Suppose  $L : \M \adjoint \N : R$ is an adjunction between (semi-)model categories.
\begin{enumerate}
\item We call $L \dashv R$ a \emph{Quillen adjunction} if the right adjoint $R$ preserves fibrations and trivial fibrations.  In this case, we call $L$ a \emph{left Quillen functor} and $R$ a \emph{right Quillen functor}.
\item We call a Quillen adjunction $L \dashv R$ a \emph{Quillen equivalence} if, for each morphism $f : LX \to Y \in \N$ with $X$ cofibrant in $\M$ and $Y$ fibrant in $\N$, $f$ is a weak equivalence in $\N$ if and only if its adjoint $f^{\#} : X \to RY$ is a weak equivalence in $\M$. 
\end{enumerate}
\end{definition}

\begin{definition}\label{pushout-corner}
Suppose $\M$ is a category with pushouts and pullbacks.
\begin{enumerate}
\item Given a solid-arrow commutative diagram
\[\begin{tikzcd}
A \ar[bend right]{ddr} \ar[bend left]{drr} \ar[densely dotted]{dr}[description]{f \pback g} &&\\
& B \timesover{D} C \ar{d} \ar{r} & C \ar{d}{g}\\
& B \ar{r}{f} & D
\end{tikzcd}\]
in $\M$ in which the square is a pullback, the unique dotted induced morphism is denoted $f \pback g$ and called the \emph{pullback corner morphism} of $f$ and $g$.
\item Given a solid-arrow commutative diagram
\[\nicexy@R-.5cm{A \ar[d] \ar[r] & C \ar[d] \ar@/^1pc/[ddr]^-{g} & \\ 
B \ar[r] \ar@/_1pc/[drr]_-{f} & B \coprodover{A} C 
 \ar@{}[dr]^(.15){}="a"^(.9){}="b" \ar@{.>} "a";"b" |-{f \pcorner g} & \\ && D}\]
in $\M$ in which the square is a pushout, the unique dotted induced morphism is denoted $f \pcorner g$ and called the \emph{pushout corner morphism} of $f$ and $g$. 
\end{enumerate}
\end{definition}

In the next definition, we follow simplicial notation $0\to 1$ so the reader can distinguish source and target at a glance.

\begin{definition}\label{mapproducts}
Suppose $(\M,\otimes,\tensorunit)$ is a monoidal category with pushouts.  Suppose $f : X_0 \to X_1$ and $g : Y_0 \to Y_1$ are morphisms in $\M$.  The pushout corner morphism 
\[\begin{tikzcd}
X_0 \otimes Y_0 \ar{d}{1 \otimes g} \ar{r}{f \otimes 1} & X_1 \otimes Y_0 \ar{d} \ar[bend left]{ddr}{1 \otimes g} &\\
X_0 \otimes Y_1 \ar[bend right=15]{drr}{f \otimes 1} \ar{r} & 
\scalebox{.8}{$(X_0\otimes Y_1) \coprodover{X_0\otimes Y_0} (X_1\otimes Y_0)$} \ar[shorten <=-3ex]{dr}[pos=.3]{f \boxprod g} &\\
&& X_1 \otimes Y_1 
\end{tikzcd}\]
of $f \otimes 1$ and $1 \otimes g$ is denoted $f \boxprod g$ and called the \emph{pushout product} of $f$ and $g$.
\end{definition}

\begin{definition}\label{ppax}
A symmetric monoidal closed category $\M$ equipped with a model structure is called a \emph{monoidal model category} if it satisfies the following \emph{pushout product axiom} \cite{ss} (3.1): 

\begin{itemize}
\item Given any cofibrations $f:X_0\to X_1$ and $g:Y_0\to Y_1$, the pushout product morphism
\[\nicexy{(X_0\otimes Y_1) \coprodover{X_0\otimes Y_0} (X_1\otimes Y_0) 
\ar[r]^-{f\boxprod g} & X_1\otimes Y_1}\]
is a cofibration. If, in addition, either $f$ or $g$ is a weak equivalence, then $f\boxprod g$ is a trivial cofibration.
\end{itemize}

Additionally, in order to guarantee that the unit $\tensorunit$ descends to the unit in the homotopy category, it is sometimes convenient to assume the \textit{unit axiom} \cite{hovey} (4.2.6): if $Q\tensorunit \to \tensorunit$ is a cofibrant replacement, then for any cofibrant object $X$, the induced morphism $Q\tensorunit \otimes X \to \tensorunit\otimes X \cong X$ is a weak equivalence. Since $(-)\otimes X$ is a left Quillen functor, if the unit axiom holds for one cofibrant replacement of $\tensorunit$, then it holds for any cofibrant replacement of $\tensorunit$.
\end{definition}

\subsection{Arrow Categories}

\begin{definition}
A \emph{lax monoidal functor} $F : \M \to \N$ between two monoidal categories is a functor equipped with structure morphisms
\[\nicexy{FX \otimes FY \ar[r]^-{F^2_{X,Y}} & F(X \otimes Y), \quad \tensorunit^{\N} \ar[r]^-{F^0} & F\tensorunit^{\M}}\]
for $X$ and $Y$ in $\M$ that are associative and unital in a suitable sense, as discussed in \cite{maclane} (XI.2), where this notion is referred to simply as a \emph{monoidal functor}. 
If, furthermore, $\M$ and $\N$ are symmetric monoidal categories, and $F^2$ is compatible with the symmetry isomorphisms, then $F$ is called a \emph{lax symmetric monoidal functor}.  If the structure morphisms $F^2$ and $F^0$ are isomorphisms (resp., identity morphisms), then $F$ is called a \emph{strong monoidal functor} (resp., \emph{strict monoidal functor}).
\end{definition}

We now recall the two monoidal structures on the arrow category from \cite{hovey-smith}.

\begin{definition}\label{def:arrowcat}
Suppose $(\M,\otimes,\tensorunit)$ is a symmetric monoidal category with pushouts. 
\begin{enumerate}
\item The \emph{arrow category} $\arrowm$ is the category whose objects are morphisms in $\M$, in which a morphism $\alpha : f \to g$ is a commutative square
\begin{equation}\label{map-in-arrowcat}
\nicexy@R-10pt{X_0 \ar[r]^-{\alpha_0} \ar[d]_f & Y_0 \ar[d]^-{g}\\ 
X_1 \ar[r]^-{\alpha_1} & Y_1}
\end{equation}
in $\M$.  We will also write $\Ev_0f = X_0$, $\Ev_1f = X_1$, $\Ev_0 \alpha = \alpha_0$, and $\Ev_1 \alpha = \alpha_1$.  The definition of $\arrowm$ does not require a monoidal structure on $\M$. 
\item The \emph{tensor product monoidal structure} on $\arrowm$ is given by the monoidal product
\[\nicexy{X_0 \otimes Y_0 \ar[r]^-{f \otimes g} & X_1 \otimes Y_1}\]
for $f : X_0 \to X_1$ and $g : Y_0 \to Y_1$.  The arrow category equipped with this monoidal structure is denoted by $\arrowmtensor$.  The monoidal unit is $\Id : \tensorunit \to \tensorunit$.
\item The \emph{pushout product monoidal structure} on $\arrowm$ is given by the pushout product
\[\nicexy{(X_0 \otimes Y_1) \coprodover{X_0 \otimes Y_0} (X_1 \otimes Y_0) \ar[r]^-{f \boxprod g} & X_1 \otimes Y_1}\]
for $f : X_0 \to X_1$ and $g : Y_0 \to Y_1$.  The arrow category equipped with this monoidal structure is denoted by $\arrowmpp$.  The monoidal unit is $\varnothing \to \tensorunit$.
\item Defining $L_0(X) = (\Id : X \to X)$ and $L_1(X) = (\varnothing \to X)$ for $X \in \M$, there are adjunctions
\begin{equation}\label{lev}
\nicexy{\M \ar@<2pt>[r]^-{L_0} & \arrowmtensor \ar@<2pt>[l]^-{\Ev_0} & \M \ar@<2pt>[r]^-{L_1} & \arrowmpp \ar@<2pt>[l]^-{\Ev_1}}
\end{equation}
with left adjoints on top and all functors strict symmetric monoidal.
\end{enumerate}
\end{definition}

\subsection{Injective Model Structure}

The following result about the injective model structure is from \cite{hovey-smith} (2.1 and 2.2).

\begin{theorem}\label{injective-model}
Suppose $\M$ is a model category.
\begin{enumerate}
\item There is a model structure on $\arrowm$, called the \emph{injective model structure}, in which a morphism $\alpha : f \to g$ as in \eqref{map-in-arrowcat} is a weak equivalence (resp., cofibration) if and only if $\alpha_0$ and $\alpha_1$ are weak equivalences (resp., cofibrations) in $\M$.  A morphism $\alpha$ is a (trivial) fibration if and only if $\alpha_1$ and the pullback corner morphism \[\nicexy{X_0 \ar[r]^-{\alpha_1 \pback\, g} & X_1 \timesover{Y_1} Y_0}\]
are (trivial) fibrations in $\M$.  Note that this implies that $\alpha_0$ is also a (trivial) fibration.  The arrow category equipped with the injective model structure is denoted by $\arrowminj$.
\item If $\M$ is cofibrantly generated, then so is $\arrowminj$.
\item If $\M$ is a monoidal model category, then $\arrowmtensor$ equipped with the injective model structure is a monoidal model category, denoted $\arrowmtensorinj$.
\item If $\M$ satisfies the unit axiom, then so does $\arrowmtensor$.
\end{enumerate}
\end{theorem}

\begin{proof}
This model structure is a special case of the injective model structure on a diagram category \cite{barwickSemi} (2.16). Since the indexing category $\bullet \to \bullet$ is so simple, we can directly write down the generating (trivial) cofibrations and hence avoid the need to assume $\M$ is combinatorial, as in \cite{white-commutative} (5.5.1). The generating cofibrations are of the form $L_1 i$ (where $i\in I$) and unit morphisms $\alpha_i: i\to U_1 \Ev_1 i$, where $U_1$ is the right adjoint of $\Ev_1$ given by $U_1(X) = 1_X$. The generating trivial cofibrations are analogous, with $j\in J$ instead of $i\in I$. A morphism $\beta: f\to g$ has the right lifting property with respect to $L_1 i$ if and only if $\Ev_1 \beta$ has the right lifting property with respect to $i$, and $\beta$ has the right lifting property with respect to $\alpha_i$ if and only if $\Ev_0 f\to \Ev_1 f \times_{\Ev_1 g} \Ev_0 g$ has the right lifting property with respect to $i$.  Thus these sets generate the injective model structure. The pushout product axiom and the unit axiom on $\arrowmtensorinj$ follows from the same on $\M$ \cite{barwickSemi} (4.51).
\end{proof}

\subsection{Projective Model Structure}

The following result about the projective model structure is from \cite{hovey-smith} (3.1).

\begin{theorem}\label{hovey-projective}
Suppose $\M$ is a model category.
\begin{enumerate}
\item There is a model structure on $\arrowm$, called the \emph{projective model structure}, in which a morphism $\alpha : f \to g$ as in \eqref{map-in-arrowcat} is a weak equivalence (resp., fibration) if and only if $\alpha_0$ and $\alpha_1$ are weak equivalences (resp., fibrations) in $\M$.  A morphism $\alpha$ is a (trivial) cofibration if and only if $\alpha_0$ and the pushout corner morphism
\[\nicexy{X_1 \coprodover{X_0} Y_0 \ar[r]^-{\alpha_1 \pcorner g} & Y_1}\]
are (trivial) cofibrations in $\M$.  Note that this implies that $\alpha_1$ is also a (trivial) cofibration.  The arrow category equipped with the projective model structure is denoted by $\arrowmproj$.
\item If $\M$ is cofibrantly generated, then so is $\arrowmproj$.
\item If $\M$ is a monoidal model category, then $\arrowmpp$ equipped with the projective model structure is a monoidal model category, denoted $\arrowmppproj$.
\item If $\M$ satisfies the unit axiom, then so does $\arrowmpp$.
\end{enumerate}
\end{theorem}

\begin{proof}
(1) and (2) follow from \cite{hirschhorn} (11.6.1). For (3), Hovey \cite{hovey-smith} (3.1) had the additional assumption that $\M$ be cofibrantly generated. However, the authors proved in \cite{white-yau-arrowcat} that if $\M$ is a monoidal model category, then so is $\arrowmppproj$. Lastly, for (4), note that a cofibrant replacement for the unit $\emptyset \to \tensorunit$ is $L_1(Q\tensorunit): \varnothing\to Q\tensorunit$. If $f$ is cofibrant in $\arrowmppproj$ (equivalently, a cofibration between cofibrant objects), then $L_1(Q\tensorunit)\boxprod f \to f$ is the same as $Q\tensorunit\otimes f \to f$.  Thus the unit axiom on $\arrowmpp$ follows from the unit axiom on $\M$.
\end{proof}

For a category $\M$ with all small limits and colimits, recall from \cite{hovey} (Sections 1.1, 6.1) that $\M$ is \textit{pointed} if the unique morphism $\emptyset \to \ast$ is an isomorphism. In such a category, we define the \textit{cokernel} of a morphism $f: X_0\to X_1$ to be the morphism $\coker f: X_1\to Z$ defined by the following pushout:
\[
\nicexy{X_0 \ar[r]^f \ar[d] & X_1 \ar[d]^{\coker f} \\ \ast \ar[r] & Z}
\]
Dually, the \textit{kernel} of $f: X_0\to X_1$ is the morphism $\ker f: A \to X_0$ defined by the following pullback:
\[
\nicexy{A \ar[r] \ar[d]_{\ker f} & \ast \ar[d] \\
X_0 \ar[r]^f & X_1}
\]
For the left adjoints $L_0$ and $L_1$ in \eqref{lev}, we note the following equalities for each object $X$.
\begin{equation}\label{Lcoker}
\begin{split}
\ker \big(L_0(X)\big) &= \ker\big(\Id : X \to X\big) = \big(\emptyset \to X\big) = L_1(X)\\
\coker \big(L_1(X)\big) &= \coker\big(\emptyset \to X\big) = \big(\Id : X \to X\big) = L_0(X)
\end{split}
\end{equation}


Most of the observations in Proposition \ref{coker-ker-pair} are from \cite{hovey-smith} (1.4, 4.1, 4.3).  We provide proofs here for completeness.

\begin{proposition}\label{coker-ker-pair}
Suppose $\M$ is a pointed symmetric monoidal category with all small limits and colimits.
\begin{enumerate}
\item\label{coker-i} The cokernel is a strictly unital strong symmetric monoidal functor from $\arrowmpp$ to $\arrowmtensor$ whose right adjoint is the kernel. 
\item\label{coker-ii} The strong symmetric monoidality of the cokernel induces a strictly unital lax symmetric monoidal structure on the kernel such that the adjunction $(\coker,\ker)$ is monoidal.
\item\label{coker-iii} If $\M$ is also a model category, then $(\coker, \ker)$ is a Quillen adjunction.
\item\label{coker-iv} If $\M$ is a stable model category \cite{hovey} (Chapter 7), then $(\coker, \ker)$ is a Quillen equivalence.
\end{enumerate}
\end{proposition}

\begin{proof}
For \eqref{coker-i}, first note that $\coker$ preserves the units since the cokernel of $\emptyset\to \tensorunit$ is $\Id_\tensorunit$. Next, it is strong monoidal because, given $f: X_0\to X_1$ and $g: Y_0\to Y_1$ we can form the following commutative diagram:
\[
\nicexy{
X_1\otimes Y_1 & X_0\otimes Y_1 \ar[r] \ar[l] & \ast \ar@{=}[d] \\
X_1\otimes Y_0 \ar@{=}[d]\ar[u] & X_0\otimes Y_0 \ar[u] \ar[d] \ar[r] \ar[l] & \ast \ar@{=}[d] \\
X_1\otimes Y_0 & X_1\otimes Y_0 \ar[r] \ar@{=}[l] & \ast \\
}
\]
Vertical pushouts yield a span whose pushout is $\coker(f\boxprod g)$.  Horizontal pushouts yield a span whose pushout is $\coker f \otimes \coker g$. Since pushouts commute, we obtain the natural isomorphism
\begin{equation}\label{cokertwo}
\begin{tikzcd}[column sep=large]
(\coker f) \otimes (\coker g) \ar{r}{\coker^2_{f,g}}[swap]{\cong} & \coker(f\boxprod g).
\end{tikzcd}
\end{equation}
We take this isomorphism as the $(f,g)$-component of the monoidal constraint for $\coker$.
Using similar reasoning and the universal property of pushouts, one can show that the symmetric monoidal coherence diagrams commute. 

For the statement that $\coker$ is left adjoint to $\ker$, note that a morphism $\alpha$ from $\coker f$ to $g$ is given by the diagram below.
\[
\nicexy{
X_0 \ar[r]^f \ar[d] & X_1 \ar[r]^{\alpha_0} \ar[d]^{\coker f} & Y_0 \ar[d]^g \\
\ast \ar[r] & Z \ar[r]_{\alpha_1} & Y_1}
\]
These data are equivalent to a morphism from $f$ to $\ker g$, since $A$ is a pullback and $Z$ is a pushout:
\[
\nicexy{
X_0 \ar[r] \ar[d]_f & A \ar[r] \ar[d]^{\ker g} & \ast \ar[d] \\
X_1 \ar[r]_{\alpha_0} & Y_0 \ar[r]_{g} & Y_1}
\]

For \eqref{coker-ii}, first note that $\ker : \arrowmtensor \to \arrowmpp$ preserves the monoidal units because the kernel of $\Id : \tensorunit \to \tensorunit$ is $\emptyset \to \tensorunit$.  The monoidal constraint of the kernel at a pair of morphisms $f$ and $g$,
\[\ker^2_{f,g} : (\ker f) \boxprod (\ker g) \to \ker(f \otimes g),\]
is adjoint to the following composite, with $\coker^2$ the monoidal constraint in \eqref{cokertwo} and $\varepsilon : \coker \circ \ker \to \Id$ the counit of the adjunction.
\begin{equation}\label{kertwo}
\begin{tikzcd}[column sep=large]
\coker\big((\ker f) \boxprod (\ker g) \big) \ar{d}{\cong}[swap]{(\coker^2)^{-1}} &\\ 
\coker(\ker f) \otimes \coker(\ker g) \ar{r}{\varepsilon_f \,\otimes\, \varepsilon_g} & f \otimes g
\end{tikzcd}
\end{equation}
The lax symmetric monoidal axioms for the kernel follow from those for the cokernel and the adjunction.

The assertion that the adjunction $(\coker,\ker)$ is monoidal means that its unit and counit are monoidal natural transformations \cite{maclane} (XI.2).  To prove this, first note that by the above description of the adjunction, its unit and counit are the identity morphisms of the monoidal units in $\arrowmpp$ and $\arrowmtensor$, respectively.

To prove that the unit $\eta : \Id \to \ker\circ\coker$ is a monoidal natural transformation, it remains to show that the following diagram commutes for each pair of morphisms $f$ and $g$.
\[\begin{tikzcd}
f \boxprod g \ar{d}[swap]{\eta_f \,\boxprod\, \eta_g} \ar{r}{\eta_{f \boxprod\, g}} & \ker\big(\coker (f \boxprod g)\big)\\
\ker(\coker f) \boxprod \ker(\coker g) \ar{r}{\ker^2} & \ker\big( \coker f \otimes \coker g\big) \ar{u}[swap]{\ker(\cong)}
\end{tikzcd}\]
This diagram commutes because the adjoint of each composite is the identity morphism of $\coker(f \boxprod g)$.  For the long composite, this uses (i) the naturality of $(\coker^2)^{-1}$ and (ii) one of the triangle identities for the adjunction $(\coker,\ker)$ \cite{maclane} (IV.1 Theorem 1).

To prove that the counit $\varepsilon : \coker \circ \ker \to \Id$ is a monoidal natural transformation, it remains to show that the following diagram commutes.
\[\begin{tikzcd}[cells={nodes={scale=.9}}, column sep=huge]
\coker(\ker f) \otimes \coker(\ker g) \ar{d}{\cong}[swap]{\coker^2} \ar{r}{\varepsilon_f \,\otimes\, \varepsilon_g} & f \otimes g\\
\coker\big(\ker f \boxprod \ker g\big) \ar{r}{\coker(\ker^2_{f,g})} & 
\coker\big(\ker(f \otimes g)\big) \ar{u}[swap]{\varepsilon_{f \otimes g}}
\end{tikzcd}\]
This diagram commutes because, starting from the lower-left corner to $f \otimes g$, each composite is adjoint to $\ker^2_{f,g}$.

For \eqref{coker-iii}, let $\alpha$ be a (trivial) cofibration and note that $\coker \alpha$ is the colimit of a morphism of pushout diagrams. That morphism of pushout diagrams is a Reedy (trivial) cofibration.  The colimit functor is left Quillen as a functor from the Reedy model structure to the underlying category \cite{hovey} (Section 5.2). Hence, $\coker \alpha$ is again a (trivial) cofibration, so $\coker$ is a left Quillen functor. See Lemma \ref{lemma-heart1} for an analogous proof.

For \eqref{coker-iv}, we must prove that, if $f$ is cofibrant in $\arrowmpp$ (so, a cofibration of cofibrant objects) and $g$ is fibrant in $\arrowmtensor$ (so, a fibration of fibrant objects), then $\alpha: \coker f \to g$ is a weak equivalence if and only if its adjoint $\beta: f\to \ker g$ is a weak equivalence \cite{hovey} (1.3.12). We display both morphisms:
\[
\nicexy{X_1 \ar[r]^{\coker f} \ar[d]_{\alpha_0} & Z \ar[d]^{\alpha_1} & & X_0 \ar[r]^f \ar[d]_{\beta_0} & X_1 \ar[d]^{\beta_1 = \alpha_0} \\
Y_0 \ar[r]_g & Y_1 & & A \ar[r]_{\ker g} & Y_0}
\]
In the homotopy category, these data give rise to fiber and cofiber sequences.  Since $\M$ is stable, every fiber sequence is canonically isomorphic to a cofiber sequence \cite{hovey} (Chapter 7).  We can extend to the right and realize $\alpha$ and $\beta$ as giving a morphism of cofiber sequences in the homotopy category:
\[
\nicexy{
X_0 \ar[r]^f \ar[d]_{\beta_0} & X_1 \ar[r]^{\coker f} \ar[d]^{\beta_1 = \alpha_0} & Z \ar[r] \ar[d]^{\alpha_1} & \Sigma X_0 \ar[d]^{\Sigma \beta_0} \\
A \ar[r]_{\ker g} & Y_0 \ar[r]_g & Y_1 \ar[r] & \Sigma A
}
\]
If either $\alpha$ or $\beta$ is a weak equivalence, then so is the other, by the two out of three property. Hence, $\coker$ and $\ker$ form a Quillen equivalence.
\end{proof}

\begin{proposition}\label{arrrowm-strongly}
Suppose $\M$ is a cofibrantly generated model category in which the domains and the codomains of all the generating cofibrations and the generating trivial cofibrations are small in $\M$.  Then $\arrowminj$ and $\arrowmproj$ are both strongly cofibrantly generated model categories.
\end{proposition}

\begin{proof}
The generating (trivial) cofibrations in $\arrowminj$ are the morphisms $L_1 i$ and the morphisms
\[\nicexy@R-.4cm{A \ar[d]_-{i} \ar[r]^-{i} & B \ar[d]^-{=}\\ B \ar[r]^-{=} & B}\]
for $i \in I$ (resp., $i \in J$) \cite{hovey-smith} (2.2).  The generating (trivial) cofibrations in $\arrowmproj$ are the morphisms $L_0I \cup L_1I$ (resp., $L_0J \cup L_1J$).  So the smallness of the domains and codomains of the generating (trivial) cofibrations in $\arrowminj$ and $\arrowmproj$ follows from our assumption on the domains and the codomains in $I$ and $J$, since a morphism in the arrow category from $f$ into a transfinite composition is determined by morphisms from $\Ev_0 f$ and $\Ev_1 f$ into transfinite compositions in $\M$.
\end{proof}

\section{Smith Ideals for Operads}\label{sec:smith-operad}

Suppose $(\M, \tensor, \tensorunit)$ is a cocomplete symmetric monoidal category in which the monoidal product commutes with colimits on both sides, which is automatically true if $\M$ is a closed symmetric monoidal category.  In this section we define Smith ideals for an arbitrary colored operad $\O$ in $\M$.  When $\M$ is pointed, we observe in Theorem \ref{thm:cok-ker-operad} that the cokernel and the kernel induce an adjunction between the categories of Smith $\O$-ideals and of $\O$-algebra morphisms.  This will set the stage for the study of the homotopy theory of Smith $\O$-ideals in the next several sections.

\subsection{Operads, Algebras, and Bimodules}

The following material on profiles and colored symmetric sequences is from \cite{yau-johnson}.  For colored operads our references are \cite{yau-operad} and \cite{white-yau}.

\begin{definition}\label{def:profiles}
Suppose $\fC$ is a set, whose elements will be called \emph{colors}.
\begin{enumerate}
\item A \emph{$\fC$-profile} is a finite, possibly empty sequence $\uc = (c_1, \ldots, c_n)$ with each $c_i \in \fC$.  
\item When permutations act on $\fC$-profiles from the left (resp., right), the resulting groupoid is denoted by $\Sigmac$ (resp., $\Sigmacop$).
\item The category of \emph{$\fC$-colored symmetric sequences} in $\M$ is the diagram category $\M^{\Sigmacopc}$.  For a $\fC$-colored symmetric sequence $X$, we think of $\Sigmacop$ (resp., $\fC$) as parametrizing the inputs (resp., outputs).  For $(\uc; d) \in \Sigmacopc$, the corresponding entry of a $\fC$-colored symmetric sequence $X$ is denoted by $X\duc$.  
\item A \emph{$\fC$-colored operad} $(\O, \gamma, 1)$ in $\M$ consists of:
\begin{itemize}
\item a $\fC$-colored symmetric sequence $\O$ in $\M$;
\item a structure morphism $\gamma : \O \circ \O \to \O$, where $\circ$ is the circle product of $\O$ in \cite{white-yau} (Definition 3.2.3), explicitly:
\[\nicexy{\O\duc \otimes \bigtensor{i=1}{n} \O\ciubi \ar[r]^-{\gamma} & \O\dub}\]
in $\M$ for all $d \in \fC$, $\uc = (c_1,\ldots,c_n) \in \Sigmac$, and $\ub_i \in \Sigmac$ for $1 \leq i \leq n$, where $\ub = (\ub_1,\ldots,\ub_n)$ is the concatenation of the $\ub_i$'s; and
\item colored units $1_c : \tensorunit \to \O\ccsingle$ for $c \in \fC$.
\end{itemize}
These data are required to satisfy the associativity, unity, and equivariant conditions in \cite{yau-operad} (Definition 11.2.1).
\item For a $\fC$-colored operad $\O$ in $\M$, an \emph{$\O$-algebra} $(A, \lambda)$ consists of:
\begin{itemize}
\item objects $A_c \in \M$ for $c \in \fC$ and
\item structure morphisms $\O \circ A \to A$, explicitly:
\[\nicexy{\O\duc \otimes A_{c_1} \otimes \cdots \otimes A_{c_n} \ar[r]^-{\lambda} & A_d}\]
in $\M$ for all $d \in \fC$ and $\uc = (c_1,\ldots,c_n) \in \Sigmac$.
\end{itemize}
These data are required to satisfy the associativity, unity, and equivariant conditions in \cite{yau-operad} (Definition 13.2.3).  Morphisms of $\O$-algebras are required to preserve the structure morphisms as in \cite{yau-operad} (Definition 13.2.8).  The category of $\O$-algebras in $\M$ is denoted by $\algom$.  The forgetful functor is denoted by $U : \algom \to \M^{\fC}$.
\item Suppose $(A,\lambda)$ is an $\O$-algebra for some $\fC$-colored operad $\O$ in $\M$.  An \emph{$A$-bimodule} $(X,\theta)$ consists of:
\begin{itemize}
\item objects $X_c \in \M$ for $c \in \fC$ and
\item structure morphisms
\[\nicexy{\O\duc \otimes A_{c_1} \otimes \cdots \otimes A_{c_{i-1}} \otimes X_{c_i} \otimes A_{c_{i+1}} \otimes \cdots \otimes A_{c_n} \ar[r]^-{\theta} & X_d}\]
in $\M$ for all $1 \leq i \leq n$ with $n \geq 1$, $d \in \fC$, and $\uc = (c_1,\ldots,c_n) \in \Sigmac$.
\end{itemize}
These data are required to satisfy associativity, unity, and equivariant conditions similar to those of an $\O$-algebra but with one input entry $A$ and the output entry replaced by $X$.  A morphism of $A$-bimodules is required to preserve the structure morphisms.  
\item For a $\fC$-colored operad $\O$ in $\M$, we write
\begin{equation}\label{Oarrows}
\arrowotensor = L_0\O \andspace \arrowopp = L_1\O
\end{equation}
for the $\fC$-colored operads in $\arrowmtensor$ and $\arrowmpp$, respectively, where $L_0 : \M \to \arrowmtensor$ and $L_1 : \M \to \arrowmpp$ are the strict monoidal functors in \eqref{lev}. 
\end{enumerate}
\end{definition}

As a consequence of \eqref{Lcoker} and \eqref{Oarrows}, we have the following equalities.
\begin{equation}\label{LofO}
\begin{split}
\ker \arrowotensor &= \ker(L_0 \O) = L_1\O = \arrowopp\\
\coker \arrowopp &= \coker (L_1 \O) = L_0\O = \arrowotensor
\end{split}
\end{equation}

\begin{definition} \label{def:operadically-cof-gen}
Suppose, moreover, that $\M$ is a cofibrantly generated model category.  We say that $\M$ is \textit{operadically cofibrantly generated} if the domains and codomains of $I$ (resp., $J$) are small with respect to a class of morphisms containing $U(\O \circ I)$-cell (resp., $U(\O\circ J)$-cell) for each $\fC$ and each $\fC$-colored operad $\O$.  More explicitly, $\O \circ - : \M^\fC \to \algom$ is a left adjoint of the forgetful functor $U$ \cite{white-yau} (4.1.11).  To form $\O \circ I$ and $\O \circ J$, we first embed $\M$ into the $c$-colored entry of $\M^\fC$ for some $c \in \fC$, with $1_\emptyset$ in all other entries, and then apply $\O \circ - $ to the images of $I$ and $J$ in $\M^\fC$.  The condition for operadically cofibrantly generated is assumed to hold for each $c \in \fC$.
\end{definition}

\begin{example} \label{example:top-operadically-cof-gen}
Every strongly cofibrantly generated model category is operadically cofibrantly generated. The category of compactly generated topological spaces is \emph{not} strongly cofibrantly generated.  However, it is operadically cofibrantly generated.  Indeed, the domains and codomains of $I\cup J$ are small relative to inclusions \cite{hovey} (2.4.1), and the morphisms in $U(\O \circ I)$-cell and $U(\O \circ J)$-cell are inclusions \cite{white-yau2} (5.10).
\end{example}

\subsection{Arrow Category of Operadic Algebras}

\begin{definition}\label{algomar}
For each $\fC$-colored operad $\O$ in $\M$, the arrow category, in the sense of Definition \ref{def:arrowcat}, of the category $\algom$ is denoted by $\algomar$.
\end{definition}

Explicitly, an object in $\algomar$ is an $\O$-algebra morphism.  A morphism in $\algomar$ is a commutative square in $\algom$ as in \eqref{map-in-arrowcat}, with each arrow an $\O$-algebra morphism.

\begin{proposition}\label{arrowotensor-algebra}
Suppose $\O$ is a $\fC$-colored operad in $\M$.  Then $\algomtensor$ is canonically isomorphic to $\algomar$.
\end{proposition}

\begin{proof}
An $\arrowotensor$-algebra $f = \{f_c : X_c \to Y_c\}$ consists of morphisms $f_c \in \M$ for $c \in \fC$ and structure morphisms
\[\nicexy{\arrowotensor\duc \otimes \bigtensor{i=1}{n} \, f_{c_i} \ar[r]^-{\lambda} & f_d}\]
in $\arrowmtensor$ for all $d \in \fC$ and $\uc = (c_1,\ldots,c_n) \in \Sigmac$.  This structure morphism is equivalent to the commutative square
\[\nicexy@R-.3cm{\O\duc \otimes \bigtensor{i=1}{n} X_{c_i} \ar[r]^-{\lambda_0} \ar[d]_-{\Id \otimes \bigotimes f_{c_i}} & X_d \ar[d]^-{f_d}\\
\O\duc \otimes \bigtensor{i=1}{n} Y_{c_i} \ar[r]^-{\lambda_1} & Y_d}\]
in $\M$.  The associativity, unity, and equivariance of $\lambda$ translate into those of $\lambda_0$ and $\lambda_1$, making $(X,\lambda_0)$ and $(Y,\lambda_1)$ into $\O$-algebras in $\M$.  The commutativity of the previous square means that $f : (X,\lambda_0) \to (Y,\lambda_1)$ is a morphism of $\O$-algebras.  The identification of morphisms in $\algomtensor$ and $\algomar$ is similarly.
\end{proof}

\begin{remark}
For the associative operad $\As$, whose algebras are monoids, the identification of $\arrowastensor$-algebras (that is, monoids in $\arrowmtensor$) with monoid morphisms in $\M$ is \cite{hovey-smith} (1.5).
\end{remark}

\subsection{Operadic Smith Ideals}

\begin{definition}\label{def:smith-ideals}
Suppose $\O$ is a $\fC$-colored operad in $\M$.  The category of \emph{Smith $\O$-ideals} in $\M$ is defined as the category $\algompp$.
\end{definition}

Propositions \ref{smith-unravel} and \ref{smith-morphism} below unpack Definition \ref{def:smith-ideals}.  They should be compared with Proposition \ref{arrowotensor-algebra}.  For objects or morphisms $A_{c_s}, \ldots, A_{c_t}$ with $s \leq t$, we use the abbreviation
\begin{equation}\label{Ast}
A_{c_{s,t}} = \bigotimes_{k=s}^t A_{c_k}.
\end{equation}

\begin{proposition}\label{smith-unravel}
Suppose $\O$ is a $\fC$-colored operad in $\M$.  A Smith $\O$-ideal in $\M$ consists of precisely
\begin{itemize}
\item an $\O$-algebra $(A,\lambda_1)$ in $\M$,
\item an $A$-bimodule $(X,\lambda_0)$ in $\M$, and
\item an $A$-bimodule morphism $f : (X,\lambda_0) \to (A,\lambda_1)$
\end{itemize}
such that, for $1 \leq i < j \leq n$, the diagram
\begin{equation}\label{two-x}
\scalebox{.9}{$
\nicexy{\O\duc \otimes A_{c_{1,i-1}} \otimes X_{c_i} \otimes A_{c_{i+1,j-1}} \otimes X_{c_j} \otimes A_{c_{j+1,n}} \ar[d]_-{\Id \otimes f_{c_i} \otimes \Id} \ar[r]^-{\Id \otimes f_{c_j} \otimes \Id} & \O\duc \otimes A_{c_{1,i-1}} \otimes X_{c_i} \otimes A_{c_{i+1,n}} \ar[d] \ar[d]^-{\lambda_0^i}\\
\O\duc \otimes A_{c_{1,j-1}} \otimes X_{c_j} \otimes A_{c_{j+1,n}} \ar[r]^-{\lambda_0^j} & X_d}$}
\end{equation}
in $\M$ is commutative.
\end{proposition}

\begin{proof}
An $\arrowopp$-algebra $(f,\lambda)$ in $\arrowmpp$ consists of
\begin{itemize}
\item morphisms $f_c : X_c \to A_c$ in $\M$ for $c \in \fC$ and
\item structure morphisms
\[\nicexy{\arrowopp\duc \boxprod f_{c_1} \boxprod \cdots \boxprod f_{c_n} \ar[r]^-{\lambda} & f_d}\]
in $\arrowmpp$ for all $d \in \fC$ and $\uc = (c_1,\ldots,c_n) \in \Sigmac$
\end{itemize}
that are associative, unital, and equivariant.  Since $\arrowopp\duc$ is the morphism $\varnothing \to \O\duc$, when $n=0$, the structure morphism $\lambda$ is simply the morphism $\lambda_1 : \O\dnothing \to A_d$ in $\M$ for $d \in \fC$.  For $n\geq 1$, the structure morphism $\lambda$ is equivalent to the commutative diagram
\begin{equation}\label{smith-structure-map}
\begin{tikzcd}
\O\duc \otimes \dom(f_{c_1} \boxprod \cdots \boxprod f_{c_n}) \ar{d}[swap]{\Id \otimes f_*} \ar{r}{\lambda_0} & X_d \ar{d}{f_d}\\
\O\duc \otimes A_{c_1} \otimes \cdots \otimes A_{c_n} \ar{r}{\lambda_1} & A_d
\end{tikzcd}
\end{equation}
in $\M$, where $f_*$ is induced by the morphisms $f_c$'s.  The bottom horizontal morphism $\lambda_1$ in \eqref{smith-structure-map} together with the morphisms $\lambda_1 : \O\dnothing \to A_d$ for $d \in \fC$ give $A$ the structure of an $\O$-algebra.  

The domain of the iterated pushout product $f_{c_1} \boxprod \cdots \boxprod f_{c_n}$ is the colimit
\begin{equation}\label{colim-punctured}
\dom(f_{c_1} \boxprod \cdots \boxprod f_{c_n}) = 
\colimover{(\epsilon_1,\ldots,\epsilon_n)}\, f_{\epsilon_1} \otimes \cdots \otimes f_{\epsilon_n}
\end{equation}
in which $(\epsilon_1,\ldots,\epsilon_n) \in \{0,1\}^n \setminus \{(1,\ldots,1)\}$ and $f_{\epsilon_i} = X_{c_i}$ (resp., $A_{c_i}$) if $\epsilon_i = 0$ (resp., $\epsilon_i = 1$).  The morphisms that define the colimit are given by the $f_{c_i}$'s.  For each $n$-tuple of indices $\epsilon = (\epsilon_1,\ldots,\epsilon_n) \in \{0,1\}^n \setminus \{(1,\ldots,1)\}$, we denote by
\begin{equation}\label{iota-epz}
\begin{tikzcd}
f_{\epsilon_1} \otimes \cdots \otimes f_{\epsilon_n} \ar{r}{\iota_\epsilon} 
& \dom(f_{c_1} \boxprod \cdots \boxprod f_{c_n})
\end{tikzcd}
\end{equation}
the morphism that comes with the colimit.  For each $i \in \{1,\ldots,n\}$, we denote by 
\[\epsilon^i = (1,\ldots,0,\ldots,1) \in \{0,1\}^n\]
the $n$-tuple with 0 in the $i$th entry and 1 in every other entry.

The top horizontal morphism $\lambda_0$ in \eqref{smith-structure-map} pre-composed with $\Id \otimes\, \iota_{\epsilon^i}$, as in
\begin{equation}\label{lambda-zero-epz}
\begin{tikzcd}[column sep=normal]
\O\duc \otimes A_{c_{1,i-1}} \otimes X_{c_i} \otimes A_{c_{i+1,n}} \ar{d}[swap]{\Id \otimes\, \iota_{\epsilon^i}} \ar[bend left=15]{dr}[pos=.6]{\lambda_0^{\epsilon^i}} &\\
\O\duc \otimes \dom(f_{c_1} \boxprod \cdots \boxprod f_{c_n}) \ar{r}[pos=.4]{\lambda_0} & X_d
\end{tikzcd}
\end{equation}
for $1 \leq i \leq n$, gives $X$ the structure of an $A$-bimodule.  The commutative diagram \eqref{smith-structure-map}, pre-composed with $\Id \otimes\, \iota_{\epsilon^i}$ as in \eqref{lambda-zero-epz}, implies that $f : (X,\lambda_0) \to (A,\lambda_1)$ is an $A$-bimodule morphism.  The morphism $\lambda_0^i$ in \eqref{two-x} is $\lambda_0^{\epsilon^i}$ in \eqref{lambda-zero-epz}.

The diagram \eqref{two-x} is the boundary of the following diagram, where $D = \dom(f_{c_1} \boxprod \cdots \boxprod f_{c_n})$.  
\begin{equation}\label{twox}
\begin{tikzcd}[column sep=0ex,cells={nodes={scale=.85}},
every label/.append style={scale=.85}]
\O\duc \otimes A_{c_{1,i-1}} \otimes X_{c_i} \otimes A_{c_{i+1,j-1}} \otimes X_{c_j} \otimes A_{c_{j+1,n}} \ar{dd}[swap]{\Id \otimes f_{c_i} \otimes \Id} \ar{rr}{\Id \otimes f_{c_j} \otimes \Id} && \O\duc \otimes A_{c_{1,i-1}} \otimes X_{c_i} \otimes A_{c_{i+1,n}} \ar{dd}{\lambda_0^{\epsilon^i}} \ar{dl}[swap,pos=.7]{\Id \otimes\, \iota_{\epsilon^i}}\\
& \O\duc \otimes D \ar{dr}{\lambda_0} &\\
\O\duc \otimes A_{c_{1,j-1}} \otimes X_{c_j} \otimes A_{c_{j+1,n}} \ar{ur}{\Id \otimes\, \iota_{\epsilon^j}} \ar{rr}{\lambda_0^{\epsilon^j}} && X_d
\end{tikzcd}
\end{equation}
The upper left quadrilateral is commutative because $D$ is the colimit in \eqref{colim-punctured}.  The other two triangles are commutative by the definition of $\lambda_0^{\epsilon^i}$ and $\lambda_0^{\epsilon^j}$ in \eqref{lambda-zero-epz}.

The argument above can be reversed.  In particular, to see that the commutative diagram \eqref{two-x}, which is the boundary of \eqref{twox}, yields the top horizontal morphism $\lambda_0$ in \eqref{smith-structure-map}, observe that the full subcategory of the punctured $n$-cube $\{0,1\}^n \setminus \{(1,\ldots,1)\}$ consisting of $(\epsilon_1,\ldots,\epsilon_n)$ with at most two $0$'s is a final subcategory \cite{maclane} (IX.3).  Thus the diagram \eqref{twox} ensures that $\lambda_0$ exists.
\end{proof}

\begin{remark}
The special case of Proposition \ref{smith-unravel} for $\O = \As$ is \cite{hovey} (1.7).
\end{remark}

\begin{proposition}\label{smith-morphism}
In the context of Proposition \ref{smith-unravel}, a morphism of Smith $\O$-ideals 
\[\Bigl(\nicexy@C-1em{(X,\lambda_0) \ar[r]^-{f} & (A,\lambda_1)}\Bigr) \overset{h}{\to} 
\Bigl(\nicexy@C-1em{(X',\lambda_0') \ar[r]^-{f'} & (A',\lambda_1')}\Bigr)\]
consists of precisely
\begin{itemize}
\item a morphism $h^1 : A \to A'$ of $\O$-algebras and
\item a morphism $h^0 : X \to X'$ of $A$-bimodules, where $X'$ becomes an $A$-bimodule via the restriction along $h^1$,
\end{itemize}
such that the square
\begin{equation}\label{fh}
\begin{tikzcd}
X_c \ar{d}[swap]{f_c} \ar{r}{h^0_c} & X'_c \ar{d}{f'_c}\\
A_c \ar{r}{h^1_c} & A'_c
\end{tikzcd}
\end{equation}
is commutative for each $c \in \fC$.
\end{proposition}

\begin{proof}
Following the proof of Proposition \ref{smith-unravel}, we unravel the given morphism $h : (f,\lambda) \to (f',\lambda')$ of $\arrowopp$-algebras.  The underlying datum of $h$ is a morphism $f \to f'$ in $\arrowm^\fC$.  Thus $h$ consists of, for each $c \in \fC$, morphisms 
\begin{equation}\label{hzerohone}
h^0_c : X_c \to X'_c \andspace h^1_c : A_c \to A'_c
\end{equation}
in $\M$ such that the square \eqref{fh} commutes.  

The compatibility of $h$ with the $\arrowopp$-algebra structure means the following diagram commutes in $\arrowm$ for all $d, c_1, \ldots, c_n \in \fC$.
\begin{equation}\label{smith-arrow-diag}
\begin{tikzcd}[column sep=large]
\arrowopp\duc \boxprod f_{c_1} \boxprod \cdots \boxprod f_{c_n} 
\ar{r}{\lambda} \ar{d}[swap]{\Id \boxprod\, h_{c_1} \boxprod\, \cdots\, \boxprod\, h_{c_n}} &
f_d \ar{d}{h_d} \\
\arrowopp\duc \boxprod f'_{c_1} \boxprod \cdots \boxprod f'_{c_n} \ar{r}{\lambda'} & f'_d 
\end{tikzcd}
\end{equation}
If $n=0$, then \eqref{smith-arrow-diag} is the commutative diagram below.
\begin{equation}\label{smith-arrow-zero}
\begin{tikzcd}[column sep=large]
\O\dnothing \ar{r}{\lambda_1} \ar[equal]{d} & A_d \ar{d}{h^1_d}\\
\O\dnothing \ar{r}{\lambda_1'}& A_d'
\end{tikzcd}
\end{equation}

For $n \geq 1$, using the abbreviation
\[D = \dom\big(f_{c_1} \boxprod \cdots \boxprod f_{c_n}\big) \andspace 
D' = \dom\big(f'_{c_1} \boxprod \cdots \boxprod f'_{c_n}\big),\] 
the diagram \eqref{smith-arrow-diag} becomes the following commutative cube.
\begin{equation}\label{smith-arrow-cube}
\begin{tikzcd}
\O\duc \otimes D \ar{dd}[swap]{\Id \otimes f_*} \ar{dr}[swap,pos=.7]{\Id \otimes\, h_*} \ar{rr}{\lambda_0} && X_d \ar{dd}[pos=.75]{f_d} \ar{dr}{h^0_d} &\\
& \O\duc \otimes D' \ar[crossing over]{rr}[pos=.25]{\lambda'_0} && X'_d \ar{dd}{f'_d}\\
\O\duc \otimes A_{c_{1,n}} \ar{dr}[swap]{\Id \otimes\, h^1_*} \ar{rr}[pos=.25]{\lambda_1} && A_d \ar{dr}[pos=.3]{h^1_d}\\
& \O\duc \otimes A'_{c_{1,n}} \ar[from=uu,crossing over, "\Id \otimes f'_*", pos=.75] \ar{rr}[pos=.4]{\lambda'_1} && A'_d
\end{tikzcd}
\end{equation}
The six commutative faces of \eqref{smith-arrow-cube} are as follows.
\begin{enumerate}
\item The back face is \eqref{smith-structure-map} for $(f,\lambda)$, expressing the $\arrowopp$-algebra structure $\lambda$ on $f$.
\item The front face is \eqref{smith-structure-map} for $(f',\lambda')$, expressing the $\arrowopp$-algebra structure $\lambda'$ on $f'$.
\item The right face is the square \eqref{fh} for $d \in \fC$.
\item The bottom face and the $n=0$ case \eqref{smith-arrow-zero} together express the fact that $h^1 : (A,\lambda_1) \to (A',\lambda'_1)$ is an $\O$-algebra morphism.
\item The left face imposes no extra condition because $D$ is the colimit in \eqref{colim-punctured} and similarly for $D'$.  In more detail, for each $n$-tuple $(\epsilon_1,\ldots,\epsilon_n) \in \{0,1\}^n \setminus \{(1,\ldots,1)\}$, the square
\begin{equation}\label{fhstar}
\begin{tikzcd}[column sep=large]
f_{\epsilon_1} \otimes \cdots \otimes f_{\epsilon_n} \ar{d}[swap]{f_*} \ar{r}{h_*} & 
f'_{\epsilon_1} \otimes \cdots \otimes f'_{\epsilon_n} \ar{d}{f'_*}\\
A_{c_1} \otimes \cdots \otimes A_{c_n} \ar{r}{h^1_*} &
A'_{c_1} \otimes \cdots \otimes A'_{c_n}
\end{tikzcd}
\end{equation}
is commutative because it is a tensor product of $n$ commutative squares corresponding to the $n$ tensor factors of the upper left corner.
\begin{itemize}
\item For a tensor factor with $\epsilon_i = 0$, by definition $f_{\epsilon_i} = X_{c_i}$ and $f'_{\epsilon_i} = X'_{c_i}$.  In this case, we have the commutative square \eqref{fh} for $c_i \in \fC$.
\item For a tensor factor with $\epsilon_i = 1$, by definition $f_{\epsilon_i} = A_{c_i}$ and $f'_{\epsilon_i} = A'_{c_i}$.  Both $f_*$ and $f'_*$ are given by the identity in the respective tensor factors, while both $h_*$ and $h^1_*$ are given by $h^1_{c_i}$.  
\end{itemize}
\end{enumerate}
Pre-composing the top face of the commutative cube \eqref{smith-arrow-cube} with the morphism $\Id \otimes\, \iota_{\epsilon^i}$ in \eqref{lambda-zero-epz} yields the following commutative diagram.
\begin{equation}\label{lambda-hzero}
\begin{tikzcd}[column sep=large]
\O\duc \otimes A_{c_{1,i-1}} \otimes X_{c_i} \otimes A_{c_{i+1,n}} 
\ar{r}{\lambda_0^{\epsilon^i}} \ar{d}[swap]{\Id \otimes\, h^0_{c_i} \otimes\, \Id} 
& X_d \ar{dd}{h^0_d}\\
\O\duc \otimes A_{c_{1,i-1}} \otimes X'_{c_i} \otimes A_{c_{i+1,n}} 
\ar{d}[swap]{\Id \otimes\, h^1_{c_{1,i-1}} \otimes\, \Id \otimes\, h^1_{c_{i+1,n}}} &\\
\O\duc \otimes A'_{c_{1,i-1}} \otimes X'_{c_i} \otimes A'_{c_{i+1,n}} 
\ar{r}{(\lambda'_0)^{\epsilon^i}} & X'_d
\end{tikzcd}
\end{equation}
This commutative diagram expresses the fact that $h^0 : X \to X'$ is a morphism of $A$-bimodules, where $X'$ becomes an $A$-bimodule via the restriction along $h^1$.

Finally, we observe that the top face of the cube \eqref{smith-arrow-cube} is actually equivalent to the commutative diagram \eqref{lambda-hzero}.  To see this, consider an $n$-tuple $\epsilon = (\epsilon_1,\ldots,\epsilon_n) \in \{0,1\}^n$ with at least two entries equal to 0.  Then the morphism $\iota_\epsilon$ in \eqref{iota-epz} factors as follows for each index $i \in \{1,\ldots,n\}$ with $\epsilon_i = 0$, and similarly for $f'$.
\[\begin{tikzcd}
f_{\epsilon_1} \otimes \cdots \otimes f_{\epsilon_n} \ar{d}[swap]{f_*} \ar{r}{\iota_\epsilon} 
& D\\
A_{c_{1,i-1}} \otimes X_{c_i} \otimes A_{c_{i+1,n}} \ar{ur}[swap]{\iota_{\epsilon^i}}
\end{tikzcd}\]
Thus, pre-composing the top face of \eqref{smith-arrow-cube} with the morphism $\Id \otimes\, \iota_{\epsilon}$ yields a diagram that factors into two sub-diagrams, one of which is \eqref{lambda-hzero}.  The other sub-diagram commutes and imposes no extra condition by the same argument above for \eqref{fhstar}.
\end{proof}

The description of Smith $\O$-ideals and their morphisms in Propositions \ref{smith-unravel} and \ref{smith-morphism} imply the following result.

\begin{proposition}\label{smitho-m}
Suppose $\O$ is a $\fC$-colored operad in $\M$.  Then there exists a $(\fC \sqcup \fC)$-colored operad $\O^s$ in $\M$ such that there is a canonical isomorphism of categories
\[\algompp \cong \algosm.\]
\end{proposition}

\begin{proof}
Denote the first and the second copies of $\fC$ in $\fC \sqcup \fC$ by, respectively, $\fC^0$ and $\fC^1$.  For an element $c \in \fC$, we write $c^\epsilon \in \fC^\epsilon$ for the same element for $\epsilon \in \{0,1\}$.  The entries of $\O^s$ are defined as follows for $d,c_1,\ldots,c_n \in \fC$ and $\epsilon_1,\ldots,\epsilon_n \in \{0,1\}$.
\[\begin{split}
\O^s\smallprof{$\binom{d^1}{c_1^{\epsilon_1}, \ldots, c_n^{\epsilon_n}}$} &= \O\duc\\
\O^s\smallprof{$\binom{d^0}{c_1^{\epsilon_1}, \ldots, c_n^{\epsilon_n}}$} &= 
\begin{cases}
\O\duc & \text{if at least one $\epsilon_i = 0$ and}\\
\emptyset & \text{otherwise.}
\end{cases}
\end{split}\]
The operad structure morphisms of $\O^s$ are either those of $\O$ or the unique morphism from the initial object $\emptyset$.  

An $\O^s$-algebra in $\M$ consists of, first of all, a $(\fC^0 \sqcup \fC^1)$-colored object in $\M$, that is, a $\fC^0$-colored object $X = \{X_c\}_{c\in\fC^0}$ and a $\fC^1$-colored object $A = \{A_c\}_{c\in\fC^1}$. 
\begin{itemize}
\item The $\O^s$-algebra structure morphism
\begin{equation}\label{lambda-one}
\nicexy{\O^s\smallprof{$\binom{d^1}{c_1^1, \ldots, c_n^1}$} \otimes A_{c_1} \otimes \cdots \otimes A_{c_n} \ar[r]^-{\lambda} & A_d}
\end{equation}
corresponds to the $\O$-algebra structure morphism $\lambda_1$ on $A$ in \eqref{smith-structure-map}. 
\item The $\O^s$-algebra structure morphism
\begin{equation}\label{lambda-zero}
\nicexy{\O^s\smallprof{$\binom{d^0}{c_1^1, \ldots, c_{i-1}^1, c_i^0, c_{i+1}^1, \ldots, c_n^1}$} \otimes A_{c_{1,i-1}} \otimes X_{c_i} \otimes A_{c_{i+1,n}} \ar[r]^-{\lambda} & X_d}
\end{equation}
corresponds to the $A$-bimodule structure morphism $\lambda_0^{\epsilon^i}$ on $X$ in \eqref{lambda-zero-epz}. 
\item The composite
\begin{equation}\label{fatd}
\nicexy@C+1em@R-1em{X_d \ar[d]_{\cong} & A_d\\
\tensorunit \otimes X_d \ar[r]^-{1_d \otimes \Id} & \O\ddsingle \otimes X_d = \O^s\smallprof{$\binom{d^1}{d^0}$} \otimes X_d \ar[u]_-{\lambda}}
\end{equation}
corresponds to the morphism $f_d$ in \eqref{smith-structure-map}.
\end{itemize} 

The identification of $\O^s$-algebra morphisms and Smith $\O$-ideal morphisms follows similarly from Proposition \ref{smith-morphism}.  More explicitly, a morphism $h$ of $\O^s$-algebras consists of a $(\fC^0 \sqcup \fC^1)$-colored morphism in $\M$.  So $h$ consists of component morphisms $h^0_c : X_c \to X'_c$ and $h^1_c : A_c \to A'_c$ as in \eqref{hzerohone}.  To see that these component morphisms make the diagram \eqref{fh} commute, we use the fact that the components of $f$ are the composites in \eqref{fatd} and similarly for $f'$.  The desired diagram \eqref{fh} is the boundary of the following diagram.
\[\begin{tikzcd}[column sep=large]
X_c \ar{d}[swap]{h^0_c} \ar{r}{\cong} & \tensorunit \otimes X_c \ar{d}[swap]{\Id \otimes\, h^0_c} \ar{r}{1_c \otimes\, \Id} & \O^s\smallprof{$\binom{c^1}{c^0}$} \otimes X_c \ar{d}[swap]{\Id \otimes\, h^0_c} \ar{r}{\lambda} & A_c \ar{d}{h^1_c}\\
X'_c \ar{r}{\cong} & \tensorunit \otimes X'_c \ar{r}{1_c \otimes\, \Id} & \O^s\smallprof{$\binom{c^1}{c^0}$} \otimes X'_c \ar{r}{\lambda'} & A_c' 
\end{tikzcd}\]
\begin{itemize}
\item The left square commutes by the naturality of the left unit isomorphism in the monoidal category $\M$.
\item The middle square commutes by the functoriality of $\otimes$.
\item The right square commutes because $h$ respects $\O^s$-algebra structures.
\end{itemize}
This shows that the diagram \eqref{fh} is commutative.

The other two conditions in Proposition \ref{smith-morphism} are the following: (i) $h^1 : A \to A'$ is an $\O$-algebra morphism.  (ii) $h^0 : X \to X'$ is an $A$-bimodule morphism.
\begin{itemize}
\item Condition (i) consists of the $n=0$ case \eqref{smith-arrow-zero} and the bottom face of the cube \eqref{smith-arrow-cube}.  These are obtained from the compatibility of $h$ with the $\O^s$-algebra structure morphism \eqref{lambda-one}.
\item Condition (ii) is the diagram \eqref{lambda-hzero}.  This is obtained from the compatibility of $h$ with the $\O^s$-algebra structure morphism \eqref{lambda-zero}.
\end{itemize}
This finishes the proof.
\end{proof}

The colored operad $\O^s$ is somewhat similar to the two-colored operad for monoid morphisms in \cite{yau-operad} (Section 14.3).

\subsection{Operadic Smith Ideals and Morphisms of Operadic Algebras}

In Proposition \ref{coker-ker-pair} we observe that, if $\M$ is a pointed symmetric monoidal category with all small limits and colimits, then there is an adjunction
\begin{equation}\label{cok-ker}
\nicexy{\arrowmpp \ar@<3pt>[r]^-{\coker} & \arrowmtensor \ar@<1pt>[l]^-{\ker}}
\end{equation}
with cokernel as the left adjoint and kernel as the right adjoint.  Since cokernel is a strictly unital strong symmetric monoidal functor, the kernel is a strictly unital lax symmetric monoidal functor, and the adjunction is monoidal.  If $\M$ is a pointed model category, then $(\coker,\ker)$ is a Quillen adjunction.  If $\M$ is a stable model category, then $(\coker,\ker)$ is a Quillen equivalence.

\begin{theorem}\label{thm:cok-ker-operad}
Suppose $\M$ is a complete and cocomplete symmetric monoidal pointed category in which the monoidal product commutes with colimits on both sides.  Suppose $\O$ is a $\fC$-colored operad in $\M$.  Then the adjunction \eqref{cok-ker} induces an adjunction
\begin{equation}\label{cok-ker-oalg}
\nicexy{\algompp \ar@<3pt>[r]^-{\coker} & \algomtensor \ar@<1pt>[l]^-{\ker}}
\end{equation}
in which the left adjoint, the right adjoint, the unit, and the counit are defined entrywise.
\end{theorem}

\begin{proof}
To simplify the notation, in this proof we write $\Cok = \coker$ and $\Ker = \ker$.  First we lift the functors $\Cok$ and $\Ker$.  Then we lift the unit and the counit for the adjunction.

\textit{Step 1: Lifting the Kernel and the Cokernel to Algebra Categories}
\smallskip

The functors in \eqref{cok-ker} lifts entrywise to the functors in \eqref{cok-ker-oalg} for the following reasons.
\begin{itemize}
\item The functor 
\[\begin{tikzcd}
\algompp & \algomtensor \ar{l}[swap]{\Ker}
\end{tikzcd}\]
exists because $\Ker : \arrowmtensor \to \arrowmpp$ is a lax symmetric monoidal functor and $\Ker \arrowotensor = \arrowopp$ by \eqref{LofO}. 
\item The functor 
\[\begin{tikzcd}
\algompp \ar{r}{\Cok} & \algomtensor
\end{tikzcd}\]
exists because $\Cok : \arrowmpp \to \arrowmtensor$ is a strong symmetric monoidal functor and $\Cok \arrowopp = \arrowotensor$ by \eqref{LofO}.
\end{itemize} 
More explicitly, suppose $(f,\lambda)$ is an $\arrowopp$-algebra as in Proposition \ref{smith-unravel}.  Then $\Cok f$ becomes an $\arrowotensor$-algebra with structure morphism $\lambda^\#$ given by the following composite for all $d,c_1,\ldots,c_n \in \fC$, with $\Cok^2 = \coker^2$ the monoidal constraint of the cokernel in \eqref{cokertwo}.
\begin{equation}\label{lambdasharp}
\begin{tikzcd}[cells={nodes={scale=.9}}]
\arrowotensor\duc \otimes \bigotimes_{i=1}^n \Cok f_{c_i} \ar[equal]{d} \ar{r}{\lambda^\#} &
\Cok f_d\\
\Cok\arrowopp\duc \otimes \bigotimes_{i=1}^n \Cok f_{c_i} \ar{r}{\Cok^2} &
\Cok\left(\arrowopp\duc \boxprod f_{c_1} \boxprod \cdots \boxprod f_{c_n} \right) \ar{u}[swap]{\Cok \lambda}
\end{tikzcd}
\end{equation}
The $\arrowotensor$-algebra axioms for $(\Cok f,\lambda^\#)$ follow from the $\arrowopp$-algebra axiom for $(f,\lambda)$ and the symmetric monoidal axioms for the cokernel.  The same reasoning also applies to the kernel.

Thus there is a diagram of functors
\begin{equation}\label{cokernel-kernel-o}
\begin{tikzcd}[column sep=large]
\algompp \ar{d}[swap]{U} \ar[shift left]{r}{\Cok} & 
\algomtensor \ar[shift left]{l}{\Ker} \ar{d}{U}\\
\arrowmppc \ar[shift left]{r}{\Cok} & \arrowmtensorc \ar[shift left]{l}{\Ker}
\end{tikzcd}
\end{equation}
with both $U$ forgetful functors and 
\[U\Ker = \Ker U.\]
To see that this equality holds, suppose $(f,\lambda)$ is an $\arrowotensor$-algebra as in the proof of Proposition \ref{arrowotensor-algebra}.  As in \eqref{lambdasharp}, the $\arrowopp$-algebra $\Ker(f,\lambda)$ is given by $(\Ker f, \lambda')$, where the $\arrowopp$-algebra structure morphism $\lambda'$ is constructed from the monoidal constraint $\Ker^2$ and $\Ker\lambda$.  Since each $U$ forgets the operad algebra structure morphism, we obtain the equalities
\[U\Ker(f,\lambda) = U(\Ker f,\lambda') = \Ker f = \Ker U(f,\lambda).\]
The equality $U\Ker = \Ker U$ holds on $\arrowotensor$-algebra morphisms because (i) both $\Ker$ apply entrywise to morphisms and (ii) both $U$ do not change the morphisms.

Next we show that the unit and the counit, 
\[\eta : \Id \to \Ker\Cok \andspace \varepsilon : \Cok \Ker \to \Id,\]
of the bottom adjunction $\Cok \dashv \Ker$ in \eqref{cokernel-kernel-o} lift to the top between algebra categories.  

\textit{Step 2: Lifting the Unit}
\smallskip

To show that $\eta$ defines a natural transformation for the top functors in \eqref{cokernel-kernel-o}, 
first we need to show that, for each $\arrowopp$-algebra $(f,\lambda)$, the unit component morphism $\eta_f : f \to \Ker\Cok f$ in $\arrowm^\fC$ is an $\arrowopp$-algebra morphism.  So we must show that the diagram \eqref{eta-alg-morphism} below in $\arrowm$ is commutative for $d,c_1,\ldots,c_n \in \fC$, with $\lambda^\#$ as in \eqref{lambdasharp}, $\eta_i = \eta_{f_{c_i}}$, $\arrowopp = \Ker \arrowotensor$ by \eqref{LofO}, and $\Ker^2 = \ker^2$ the monoidal constraint defined in \eqref{kertwo}.
\begin{equation}\label{eta-alg-morphism}
\begin{tikzcd}[column sep=large]
\arrowopp\duc \boxprod \lbox_{i=1}^n f_{c_i} \ar{d}[swap]{\Id \boxprod\, \lbox_i \eta_i} \ar{r}{\lambda} & f_d \ar{dd}{\eta_{f_d}}\\
\Ker\arrowotensor \duc \boxprod \lbox_{i=1}^n \Ker\Cok f_{c_i} \ar{d}[swap]{\Ker^2} &\\
\Ker\left(\arrowotensor\duc \otimes \bigotimes_{i=1}^n \Cok f_{c_i} \right) \ar{r}{\Ker \lambda^\#} & \Ker\Cok f_d
\end{tikzcd}
\end{equation}
To see that \eqref{eta-alg-morphism} is commutative, we consider the adjoint of each composite, which yields the boundary of the following diagram in $\arrowm$.
\begin{equation}\label{eta-alg-adjoint}
\begin{tikzcd}[column sep=-2em,cells={nodes={scale=.85}}]
\Cok\left( \arrowopp\duc \boxprod \lbox_{i=1}^n f_{c_i}\right) 
\ar{dd}[swap]{\Cok(\Id \boxprod \, \lbox_i \eta_i)} \ar{dr}[pos=.75]{(\Cok^2)^{-1}} 
\ar{rrr}{\Cok\lambda} &&& \Cok(f_d)\\
& \Cok\arrowopp\duc \otimes \bigotimes_{i=1}^n \Cok f_{c_i} 
\ar{dr}[swap,pos=.3]{\Id \otimes \bigotimes_i \Cok\eta_i} \ar[equal]{rr} 
&& \arrowotensor\duc \otimes \bigotimes_{i=1}^n \Cok f_{c_i} \ar{u}[swap]{\lambda^\#}\\
\Cok\left( \Ker\arrowotensor\duc \boxprod \lbox_{i=1}^n \Ker\Cok f_{c_i}\right) 
\ar{rr}[swap]{(\Cok^2)^{-1}} &&
\Cok\Ker\arrowotensor\duc \otimes \bigotimes_{i=1}^n \Cok\Ker\Cok f_{c_i} 
\ar{ur}[swap,pos=.75]{\varepsilon_{\arrowotensor\duc} \otimes\, \bigotimes_i \varepsilon_{\Cok f_{c_i}}} &
\end{tikzcd}
\end{equation}
The three sub-regions in \eqref{eta-alg-adjoint} are commutative for the following reasons.
\begin{itemize}
\item The left triangle is commutative by the naturality of the monoidal constraint $\Cok^2 = \coker^2$ of the cokernel.
\item The upper right region is commutative by the definition of $\lambda^\#$ in \eqref{lambdasharp}.
\item To see that the lower right triangle is commutative, first note that the counit component morphism
\begin{equation}\label{epz-o-id}
\varepsilon_{\arrowotensor\duc} : \Cok\Ker\arrowotensor\duc \to \arrowotensor\duc
\end{equation}
is the identity, since by \eqref{LofO}
\[\Cok\Ker\arrowotensor = \Cok\arrowopp = \arrowotensor.\]
For each of the other $n$ tensor factors in the lower right triangle, the composite $\varepsilon_{\Cok f_{c_i}} \circ \Cok\eta_i$ is the identity morphism by one of the triangle identities for the adjunction $\Cok \dashv \Ker$ \cite{maclane} (IV.1 Theorem 1).
\end{itemize}
This proves that $\eta_f : f \to \Ker\Cok f$ is an $\arrowopp$-algebra morphism.  Moreover, $\eta$ is natural with respect to $\arrowopp$-algebra morphisms because a diagram in $\algompp$ is commutative if and only if its underlying diagram in $\arrowm^\fC$ is commutative.  Thus the unit $\eta : \Id \to \Ker\Cok$ is a natural transformation for the top horizontal functors in \eqref{cokernel-kernel-o} between algebra categories.

\textit{Step 3: Lifting the Counit}
\smallskip

Next we show that the counit $\varepsilon : \Cok\Ker \to \Id$ of the bottom adjunction $\Cok \dashv \Ker$ in \eqref{cokernel-kernel-o} lifts to the top between algebra categories.  First we need to show that, for each $\arrowotensor$-algebra $(g,\lambda)$, the counit component morphism $\varepsilon_g : \Cok\Ker g \to g$ in $\arrowm^\fC$ is an $\arrowotensor$-algebra morphism.  Denote by 
\[(\Ker g, \labar) = \Ker(g,\lambda)\] 
the $\arrowopp$-algebra obtained by applying the top functor $\Ker$ in \eqref{cokernel-kernel-o}.  The $\arrowopp$-algebra structure morphism $\labar$ is the analogue of \eqref{lambdasharp} for the kernel.  In other words, it is the composite
\begin{equation}\label{labardef}
\labar = (\Ker \lambda) \circ \Ker^2
\end{equation}
with $\Ker^2 = \ker^2$ the monoidal constraint \eqref{kertwo}.

As noted in \eqref{epz-o-id}, each component $\varepsilon_{\arrowotensor\duc}$ is the identity.  With $(-)^\#$ as in \eqref{lambdasharp}, $\varepsilon_g$ is an $\arrowotensor$-algebra morphism if and only if the boundary of the diagram \eqref{epz-alg-morphism} below in $\arrowm$ commutes for $d,c_1,\ldots,c_n \in \fC$, with $\varepsilon_i = \varepsilon_{g_{c_i}}$, $\arrowotensor = \Cok\arrowopp$, and $\arrowopp = \Ker\arrowotensor$ by \eqref{LofO}.
\begin{equation}\label{epz-alg-morphism}
\begin{tikzcd}[column sep=-2em,cells={nodes={scale=.9}}]
\arrowotensor\duc \otimes \bigotimes_{i=1}^n \Cok\Ker g_{c_i} \ar{ddd}[swap]{\Id \otimes \bigotimes_i \varepsilon_i} \ar{dr}[pos=.75]{\Cok^2} \ar{rrr}{\labar^\#} &&& \Cok\Ker g_d \ar[bend left=15]{ddd}{\varepsilon_{g_d}}\\
& \Cok\left( \arrowopp\duc \boxprod \lbox_{i=1}^n \Ker g_{c_i}\right) \ar{dr}[pos=.7]{\Cok(\Ker^2)} \ar{urr}[pos=.25]{\Cok \labar} &&\\
&& \Cok\Ker\left(\arrowotensor\duc \otimes \bigotimes_{i=1}^n g_{c_i}\right) \ar{dll}[swap,pos=.6]{\varepsilon} \ar{uur}[pos=.5]{\Cok\Ker \lambda} &\\
\arrowotensor\duc \otimes \bigotimes_{i=1}^n g_{c_i} \ar{rrr}{\lambda} &&& g_d
\end{tikzcd}
\end{equation}
The four sub-regions in \eqref{epz-alg-morphism} are commutative for the following reasons.
\begin{itemize}
\item The top triangle is commutative by the definition of $(-)^\#$ in \eqref{lambdasharp}.
\item The triangle to its lower right is commutative by the definition of $\labar$ in \eqref{labardef} and the functoriality of $\Cok$.
\item The lower right quadrilateral is commutative by the naturality of the counit $\varepsilon : \Cok\Ker \to \Id$.
\end{itemize}
Using the inverse of $\Cok^2 = \coker^2$, the left triangle in \eqref{epz-alg-morphism} is equivalent to the following diagram.
\begin{equation}\label{epz-alg-diag}
\begin{tikzcd}[column sep=large]
\Cok\left( \Ker\arrowotensor\duc \boxprod \lbox_{i=1}^n \Ker g_{c_i}\right) \ar{d}[swap]{(\Cok^2)^{-1}} \ar{r}{\Cok(\Ker^2)} & 
\Cok\Ker\left(\arrowotensor\duc \otimes \bigotimes_{i=1}^n g_{c_i}\right) \ar{d}{\varepsilon}\\
\Cok\Ker\arrowotensor\duc \otimes \bigotimes_{i=1}^n \Cok\Ker g_{c_i} \ar{r}{\varepsilon \,\otimes\, \bigotimes_i \varepsilon_i} &
\arrowotensor\duc \otimes \bigotimes_{i=1}^n g_{c_i} 
\end{tikzcd}
\end{equation}
The diagram \eqref{epz-alg-diag} is commutative because the adjoint of each composite is $\Ker^2 = \ker^2$ defined in \eqref{kertwo}.  This shows that \eqref{epz-alg-morphism} is commutative, and $\varepsilon_g$ is an $\arrowotensor$-algebra morphism.  

Moreover, $\varepsilon$ is natural with respect to $\arrowotensor$-algebra morphisms because a diagram in $\algomtensor$ is commutative if and only if its underlying diagram in $\arrowm^\fC$ is commutative.  Thus the counit $\varepsilon : \Cok\Ker \to \Id$ is a natural transformation for the top horizontal functors in \eqref{cokernel-kernel-o} between algebra categories.

Finally, the lifted natural transformations $\eta$ and $\varepsilon$ satisfy the triangle identities for an adjunction \cite{maclane} (IV.1 Theorem 1) because diagrams in $\algompp$ and $\algomtensor$ are commutative if and only if their underlying diagrams in $\arrowm^\fC$ are commutative.  This proves that the top horizontal functors $(\Cok,\Ker)$ in \eqref{cokernel-kernel-o} form an adjunction with the lifted unit and counit.
\end{proof}


\section{Homotopy Theory of Smith Ideals for Operads}
\label{sec:homotopy-smith}

In this section, we study the homotopy theory of Smith ideals for an operad $\O$.  Under suitable conditions on the underlying monoidal model category $\M$, in Def. \ref{def:smitho-model} we define model structures on the categories of Smith $\O$-ideals and of $\O$-algebra morphisms.  When $\M$ is pointed, the cokernel and the kernel yield a Quillen adjunction between these model categories.  Furthermore, in Theorem \ref{smith=map} we show that if $\M$ is stable and if cofibrant Smith $\O$-ideals are entrywise cofibrant in $\arrowmppproj$, then the cokernel and the kernel yield a Quillen equivalence between the categories of Smith $\O$-ideals and of $\O$-algebra morphisms.

\begin{definition}\label{def:admissibleO}
We say that a $\fC$-colored operad $\O$ is \emph{admissible} if $\algom$ admits a transferred model structure, with weak equivalences and fibrations defined entrywise in $\M^\fC$. 
\end{definition}

\subsection{Admissibility of Operads}
The following result is \cite{white-yau} (6.1.1 and 6.1.3).

\begin{theorem}\label{spade}
Suppose $\M$ is an operadically cofibrantly generated (Def. \ref{def:operadically-cof-gen}) monoidal model category satisfying the following condition.
\begin{quote}
$(\spadesuit)$ : For each $n \geq 1$ and for each object $X \in \Msigmanop$, the function
\[X \tensorover{\Sigma_n} (-)^{\boxprod n} : \M \to \M\]
takes trivial cofibrations into some subclass of weak equivalences that is closed under transfinite composition and pushout.
\end{quote}
Then each $\fC$-colored operad $\O$ in $\M$ is admissible in the sense of Definition \ref{def:admissibleO}.
\end{theorem}

\begin{example}\label{spade-examples}
Strongly cofibrantly generated monoidal model categories that satisfy $(\spadesuit)$ include:
\begin{enumerate}
\item Pointed or unpointed simplicial sets \cite{quillen} and all of their left Bousfield localizations \cite{hirschhorn}.
\item Bounded or unbounded chain complexes over a commutative ring containing the rationals $\mathbb{Q}$ \cite{quillen}.
\item Symmetric spectra built on either simplicial sets or compactly generated topological spaces, motivic symmetric spectra, and $G$-equivariant symmetric spectra with either the positive stable model structure or the positive flat stable model structure \cite{dmitri}.
\item The category of small categories with the folk model structure \cite{rezk}.
\item Simplicial modules over a field of characteristic zero \cite{quillen}.
\item The stable module category of $k[G]$-modules \cite{hovey} (2.2), where $k$ is a field and $G$ is a finite group. 
We recall that the homotopy category of this example is trivial unless the characteristic of $k$ divides the order of $G$ (the setting for \textit{modular} representation theory).
\end{enumerate}
The condition $(\spadesuit)$ for (1)--(2) is proved in \cite{white-yau} (Section 8, which also handles symmetric spectra built on simplicial sets), and (4)--(5) can be proved using similar arguments.  The condition $(\spadesuit)$ for the stable module category is proved by the argument in \cite{white-yau2} (12.2).  For symmetric spectra built on topological spaces, motivic symmetric spectra, and equivariant symmetric spectra, we refer to \cite{dmitri} (Section 2, and the references therein) starting with $\cat C = Top$, $sSet^G$, $Top^G$, and the $\mathbb{A}^1$-localization of simplicial presheaves with the injective model structure. 

In each of these examples except those built from $Top$, the domains and the codomains of the generating (trivial) cofibrations are small with respect to the entire category.  So Proposition \ref{arrrowm-strongly} applies to show that, in each case, the arrow category with either the injective or the projective model structure is strongly cofibrantly generated. The category of (equivariant) symmetric spectra built on topological spaces is operadically cofibrantly generated by an argument analogous to that of Example \ref{example:top-operadically-cof-gen}, as are the arrow categories, by the remark below.
\end{example}

\begin{remark}
In \cite{white-yau} (6.1.1 and 6.1.3), $\M$ is assumed to be strongly cofibrantly generated, but actually operadically cofibrantly generated suffices for the proof. The smallness hypothesis is required in order to run the small object argument, and $\O \circ I$ (resp. $\O \circ J$) are the generating (trivial) cofibrations. We have previously pointed out that \emph{operadically cofibrantly generated} is a sufficient smallness hypothesis in \cite{white-yau2} (5.7). The proof of Proposition \ref{arrrowm-strongly} also proves that, if $\M$ is operadically cofibrantly generated, then so are $\arrowmpp$ and $\arrowmtensor$.
\end{remark}

Even if $(\spadesuit)$ is not satisfied, sometimes the classes of morphisms defined in Theorem \ref{spade} in $\algom$ define a semi-model structure \cite{white-yau} (6.2.3 and 6.3.1). We therefore phrase our arguments in this section to only rely on the semi-model category axioms in categories of algebras. In Section \ref{sec:appendix}, we include a comparison to the $\infty$-categorical approach to encoding the homotopy theory of operad-algebras.

\subsection{Admissibility of Operads in the Arrow Category}

Recall the injective model structure on the arrow category, which is a monoidal model category if $\M$ is, by Theorem \ref{injective-model}.

\begin{theorem}\label{spade-arrow-tensor}
If $\M$ is a monoidal model category satisfying $(\spadesuit)$, then so is $\arrowmtensorinj$.  Therefore, if $\M$ is also cofibrantly generated in which the domains and the codomains of all the generating (trivial) cofibrations are small in $\M$, then every $\fC$-colored operad on $\arrowmtensorinj$ is admissible.
\end{theorem}

\begin{proof}
Suppose $\M$ satisfies $(\spadesuit)$ with respect to a subclass $\C$ of weak equivalences that is closed under transfinite composition and pushout.  We write $\C'$ for the subclass of weak equivalences $\beta$ in $\arrowmtensorinj$ such that $\beta_0,\beta_1 \in \C$.  Then $\C'$ is closed under transfinite composition and pushout.

Suppose $f_X : X_0 \to X_1$ is an object in $\arrowmtensorsigmanop$ and $\alpha : f_V \to f_W$,
\begin{equation}\label{alphavw}
\nicexy@R-.3cm{V_0 \ar[d]_-{f_V} \ar[r]^-{\alpha_0} & W_0 \ar[d]^-{f_W}\\ V_1 \ar[r]^-{\alpha_1} & W,}
\end{equation}
is a trivial cofibration in $\arrowmtensorinj$.  We will show that $f_X \tensor_{\Sigma_n} \alpha^{\boxprod n}$ belongs to $\C'$.  The morphism $f_X \tensor_{\Sigma_n} \alpha^{\boxprod n}$ in $\arrowmtensor$ is the commutative square
\[\nicexy@C+.4cm{X_0 \tensorover{\Sigma_n} \dom(\alpha_0^{\boxprod n}) \ar[d]_-{f_X \tensorover{\Sigma_n} f_*} \ar[r]^-{X_0 \tensorover{\Sigma_n} \alpha_0^{\boxprod n}} & X_0 \tensorover{\Sigma_n} W_0^{\tensor n} \ar[d]^-{f_X \tensorover{\Sigma_n} f_W^{\tensor n}}\\
X_1 \tensorover{\Sigma_n} \dom(\alpha_1^{\boxprod n}) \ar[r]^-{X_1 \tensorover{\Sigma_n} \alpha_1^{\boxprod n}} & X_1 \tensorover{\Sigma_n} W_1^{\tensor n}}\]
in $\M$, where $f_*$ is induced by $f_V$ and $f_W$.  Since $\alpha_0$ and $\alpha_1$ are trivial cofibrations in $\M$ and since $X_0, X_1 \in \Msigmanop$, the condition $(\spadesuit)$ in $\M$ implies that the two horizontal morphisms in the previous diagram are both in $\C$.  This shows that $\arrowmtensorinj$ satisfies $(\spadesuit)$ with respect to the subclass $\C'$ of weak equivalences.  

The second assertion is now a consequence of Proposition \ref{arrrowm-strongly}, Example \ref{example:top-operadically-cof-gen}, and Theorem \ref{spade}.
\end{proof}

\begin{definition}\label{def:smitho-model}
Suppose $\M$ is a cofibrantly generated monoidal model category satisfying $(\spadesuit)$ in which the domains and the codomains of the generating (trivial) cofibrations are small with respect to the entire category.  Suppose $\O$ is a $\fC$-colored operad in $\M$.
\begin{enumerate}
\item Equip the category of Smith $\O$-ideals $\algompp$ with the model structure given by Corollary \ref{smitho-m} and Theorem \ref{spade}.  In other words, a morphism $\alpha$ of Smith $\O$-ideals is a weak equivalence (resp., fibration) if and only if $\alpha_0$ and $\alpha_1$ are color-wise weak equivalences (resp., fibrations) in $\M$.
\item Equip the category $\algomtensor$ with the model structure given by Theorem \ref{spade-arrow-tensor}.  In other words, a morphism $\alpha$ in $\algomtensor$ is a weak equivalence (resp., fibration) if and only if $\alpha^c$ ($=$ the $c$-colored entry of $\alpha$) is a weak equivalence (resp., fibration) in $\arrowmtensorinj$ for each $c \in \fC$.  
\end{enumerate}

When $(\spadesuit)$ is not satisfied but the classes of morphisms above still define semi-model structures (e.g., Remark \ref{remark:semi-com}, Corollary \ref{sigmacof-smith=map}, and Theorem \ref{underlying-cofibrant}), we still denote those semi-model structures by $\algompp$ and $\algomtensor$. 
\end{definition}

\begin{remark}
Recall diagram (\ref{cokernel-kernel-o}).
In Definition \ref{def:smitho-model} the (semi-)model structure on Smith $\O$-ideals is induced by the forgetful functor to $\M^{\fC \sqcup \fC}$, so its weak equivalences and fibrations are defined entrywise in $\M$, or equivalently in $\arrowmppproj$.  On the other hand, the model structure on $\O$-algebra morphisms $\algomtensor$ is induced by the forgetful functor to $(\arrowmtensorinj)^{\fC}$.  The (trivial) fibrations in $\algomtensor$ are, in particular, entrywise (trivial) fibrations in $\M$.  However, they are \emph{not} defined entrywise in $\M$, since (trivial) fibrations in $\arrowmtensorinj$ are not defined entrywise in $\M$, as explained in Theorem \ref{injective-model}.
\end{remark}

Suppose $K \subseteq \M$ is a subclass of morphisms in a category $\M$ with a chosen initial object and $\fC$ is a set with $c \in \fC$.  We denote by 
\[K_c \subseteq \M^\fC\]
the subclass of morphisms in which the morphisms in $K$ are concentrated in the $c$-entry with all other entries the initial object.  The following observation will be used in the proof of Theorem \ref{underlying-cofibrant} below.

\begin{proposition}\label{cgoverarrowm}
In the context of Definition \ref{def:smitho-model}, the (semi-)model structure on Smith $\O$-ideals is  cofibrantly generated with generating cofibrations $\arrowopp \circ (L_0I \cup L_1I)_c$ and generating trivial cofibrations $\arrowopp \circ (L_0J \cup L_1J)_c$ for $c \in \fC$, where $I$ and $J$ are the sets of generating cofibrations and generating trivial cofibrations in $\M$.  
\end{proposition}

\begin{proof}
The category $\algompp$ already has a (semi-)model structure, namely, the one in Def. \ref{def:smitho-model}(1), with weak equivalences and fibrations defined via the forgetful functor $U$ in the free-forgetful adjunction
\[\nicexy@C+1em{(\arrowmppproj)^{\fC} \ar@<2pt>[r]^-{\arrowopp \circ -} & \algompp, \ar@<2pt>[l]^-{U}}\]
since the weak equivalences and fibrations in $\arrowmproj$ are defined in $\M$. To see that $\algompp$ has a cofibrantly generated model structure with weak equivalences and fibrations defined entrywise in $\arrowmproj$ and with generating (trivial) cofibrations as stated above, we refer to the computations of Lemma 3.3 in \cite{johnson-yau}, which produces the sets $I$ and $J$, proves the requisite smallness, and proves that fibrations and trivial fibrations are characterized by lifting with respect to $I$ and $J$. Hence, this proof works just as well for semi-model categories. Since a (semi-)model structure is uniquely determined by the classes of weak equivalences and fibrations, this second model structure on $\algompp$ must be the same as the one in Definition \ref{def:smitho-model} (1).
\end{proof}

\subsection{Quillen Adjunction Between Operadic Smith Ideals and Algebra Morphisms}

\begin{proposition}\label{prop:smith-map-qadjunction}
Suppose $\M$ is a pointed cofibrantly generated monoidal model category, in which the domains and the codomains of the generating (trivial) cofibrations are small with respect to the entire category.  Suppose $\O$ is a $\fC$-colored operad in $\M$ such that $\algompp$ and $ \algomtensor$ admit transferred semi-model structures as in Definition \ref{def:smitho-model}.  Then the adjunction
\begin{equation}\label{smith-map-qadjunction}
\nicexy{\algompp \ar@<3pt>[r]^-{\coker} & \algomtensor \ar@<1pt>[l]^-{\ker}}
\end{equation}
in \eqref{cok-ker-oalg} is a Quillen adjunction.
\end{proposition}

\begin{proof}
Suppose $\alpha$ is a (trivial) fibration in $\algomtensor$.  We must show that $\ker\alpha$ is a (trivial) fibration in $\algompp$, that is, an entrywise (trivial) fibration in $\M$.  Since (trivial) fibrations in $\arrowmppproj$ are defined entrywise in $\M$, it suffices to show that $U\ker\alpha$ is a (trivial) fibration in $(\arrowmppproj)^{\fC}$.  Since there is an  equality \eqref{cokernel-kernel-o}
\[U\ker\alpha = \ker U\alpha\]
and since $\ker : (\arrowmtensorinj)^{\fC} \to (\arrowmppproj)^{\fC}$ is a right Quillen functor by Proposition \ref{coker-ker-pair} \eqref{coker-iii}, we finish the proof by observing that $U\alpha \in (\arrowmtensorinj)^{\fC}$ is a (trivial) fibration.
\end{proof}

Recall that a pointed (semi-)model category is \emph{stable} if its homotopy category is a triangulated category \cite{hovey} (7.1.1).

\begin{proposition}\label{reflects-weq}
In the setting of Proposition \ref{prop:smith-map-qadjunction}, suppose $\M$ is also a stable (semi-)model category.  Then the right Quillen functor $\ker$ in \eqref{smith-map-qadjunction} reflects weak equivalences between fibrant objects.
\end{proposition}

\begin{proof}
Suppose $\alpha$ is a morphism in $\algomtensor$ between fibrant objects such that $\ker \alpha \in \algompp$ is a weak equivalence.  So $\ker \alpha$ is entrywise a weak equivalence in $\M$, or equivalently $U\ker\alpha \in (\arrowmppproj)^{\fC}$ is a weak equivalence.  We must show that $\alpha$ is a weak equivalence, that is, that $U\alpha \in (\arrowmtensorinj)^{\fC}$ is a weak equivalence.  The morphism $U\alpha$ is still a morphism between fibrant objects, and 
\[\ker U\alpha = U\ker \alpha\]
is a weak equivalence in $(\arrowmppproj)^{\fC}$.  Since $\ker : (\arrowmtensorinj)^{\fC} \to (\arrowmppproj)^{\fC}$ is a right Quillen equivalence by Proposition \ref{coker-ker-pair} \eqref{coker-iv}, it reflects weak equivalences between fibrant objects by \cite{hovey} (1.3.16).  So $U\alpha$ is a weak equivalence.
\end{proof}

\subsection{Quillen Equivalence Between Operadic Smith Ideals and Algebra Morphisms}

The following result says that, under suitable conditions, Smith $\O$-ideals and $\O$-algebra morphisms have equivalent homotopy theories.

\begin{theorem}\label{smith=map}
Suppose $\M$ is a cofibrantly generated stable monoidal model category, and $\algompp$ and $\algomtensor$ admit transferred (semi-)model structures as in Definition \ref{def:smitho-model}.  Suppose $\O$ is a $\fC$-colored operad in $\M$ such that cofibrant $\arrowopp$-algebras are also underlying cofibrant in $(\arrowmppproj)^{\fC}$.  Then the Quillen adjunction
\[\nicexy{\algompp \ar@<3pt>[r]^-{\coker} & \algomtensor \ar@<1pt>[l]^-{\ker}}\]
is a Quillen equivalence.
\end{theorem}

\begin{proof}
Using Proposition \ref{reflects-weq} and \cite{hovey} (1.3.16) (or \cite{white-commutative} (4.3) for the semi-model category case), it remains to show that for each cofibrant object $f_X \in \algompp$, the derived unit
\[\nicexy{f_X \ar[r]^-{\eta} & \ker \RO \coker f_X}\]
is a weak equivalence in $\algompp$, where $\RO$ is a fibrant replacement functor in $\algomtensor$.  In other words, we must show that $U\eta$ is a weak equivalence in the model category $(\arrowmppproj)^{\fC}$.

Suppose $R$ is a fibrant replacement functor in $(\arrowmtensorinj)^{\fC}$.  Consider the solid-arrow commutative diagram
\[\nicexy{U \coker f_X \ar@{>->}[d]_-{\sim} \ar[r]^-{\sim} & U\RO \coker f_X \ar@{>>}[d]\\ RU\coker f_X \ar@{>>}[r] \ar@{.>}[ur]|-{\makebox[2\width]{$\alpha$}} & 0}\]
in $(\arrowmtensorinj)^{\fC}$.  Here the left vertical morphism is a trivial cofibration and is a fibrant replacement of $U\coker f_X$.  The top horizontal morphism is a weak equivalence and is $U$ applied to a fibrant replacement of $\coker f_X$.  The other two morphisms are fibrations.  So there is a dotted morphism $\alpha$ that makes the whole diagram commutative.  By the $2$-out-of-$3$ property, $\alpha$ is a weak equivalence between fibrant objects in $(\arrowmtensorinj)^{\fC}$.  Since $\ker : (\arrowmtensorinj)^{\fC} \to (\arrowmppproj)^{\fC}$ is a right Quillen functor, by Ken Brown's Lemma \cite{hovey} (1.1.12) $\ker\alpha$ is a weak equivalence in $(\arrowmppproj)^{\fC}$.

We now have a commutative diagram
\[\nicexy{Uf_X \ar[d]_-{\varepsilon} \ar[rr]^-{U\eta} && U\ker\RO\coker f_X\\
\ker R\coker Uf_X \ar[r]^-{=} & \ker RU\coker f_X \ar[r]^-{\ker\alpha}_-{\sim} & \ker U\RO\coker f_X\ar[u]_-{=}}\]
in $(\arrowmppproj)^{\fC}$, where $\varepsilon$ is the derived unit of $Uf_X$.  To show that $U\eta$ is a weak equivalence, it suffices to show that $\varepsilon$ is a weak equivalence. By assumption $Uf_X$ is a cofibrant object in $(\arrowmppproj)^{\fC}$.  Since $(\coker,\ker)$ is a Quillen equivalence between $(\arrowmppproj)^{\fC}$ and $(\arrowmtensorinj)^{\fC}$, the derived unit $\varepsilon$ is a weak equivalence by \cite{hovey} (1.3.16).
\end{proof}

\begin{example}\label{stable-examples}
Among the model categories in Example \ref{spade-examples}, 
\begin{enumerate}
\item the categories of bounded or unbounded chain complexes over a semi-simple ring that contains the rational numbers, 
\item the stable module category of $k[G]$-modules,
\item the categories of symmetric spectra, $G$-equivariant symmetric spectra built on simplicial sets for a finite group $G$, and motivic symmetric spectra, with either the positive or the positive flat stable model structure
\end{enumerate}
satisfy the conclusion of Theorem \ref{smith=map}, for every operad $\O$.  Admissibility is proven in \cite{white-yau} (6.1.1) and \cite{white-yau2} (5.15). Stability is discussed in \cite{hovey} (Chapter 7), \cite{white-yau} (8.3), and \cite{dmitri} (Section 2). All are strongly cofibrantly generated because they are combinatorial model categories \cite{white-yau2} (Sections 11 and 12), \cite{dmitri} (Section 2). So all satisfy the conditions of Theorem \ref{smith=map} except that the condition about cofibrant Smith $\O$-ideals being color-wise cofibrant in $\arrowmppproj$ is more subtle. We will consider this issue in the next two sections, proving this condition for (1) in Corollary \ref{chain-zero} and for (2) in Corollary \ref{stmod-alloperad}.

For classical, equivariant, or motivic symmetric spectra, we must tweak the proof of Theorem \ref{smith=map}. Let $(\arrowmppproj)^\fC$ refer to the projective model structure on the arrow category where $\M$ is the \textit{injective} stable model structure on the relevant category of symmetric spectra. Since the weak equivalences of the injective stable model structure coincide with those of the positive (flat) stable model structure, in the last paragraph of the proof, it is enough to prove that $\epsilon$ is a weak equivalence with respect to the injective stable model structure on spectra. Hence, it suffices for $Uf_X$ to be a cofibrant object in $(\arrowmppproj)^\fC$, which follows from the proof of \cite{white-yau} (8.3.3), using our filtrations and the fact that the cofibrations of the injective stable model structure are the monomorphisms.
\end{example}




We note that we cannot add the injective stable model structure on symmetric spectra to the list in Example \ref{stable-examples} because it is not true that every operad is admissible. A famous obstruction due to Gaunce Lewis prevents the $\Com$ operad from being admissible, for example.

\section{Smith Ideals for Commutative and Sigma-Cofibrant Operads}\label{sec:com}

In this section we apply Theorem \ref{smith=map} and consider Smith ideals for the commutative operad and $\Sigmac$-cofibrant operads (Definition \ref{def:sigma-cof}).  In particular, in Corollary \ref{sigmacof-smith=map} we will show that Theorem \ref{smith=map} is applicable to all $\Sigmac$-cofibrant operads.  On the other hand, the commutative operad is usually not $\Sigma$-cofibrant.  However, as we will see in Example \ref{com-spectra}, Theorem \ref{smith=map} is applicable to the commutative operad in symmetric spectra with the positive flat stable model structure.

\subsection{Commutative Smith Ideals} \label{subsec:comm}

For the commutative operad, which is entrywise the monoidal unit and whose algebras are commutative monoids, we use the following definition from \cite{white-commutative} (3.4).  The notation $?/{\Sigma_n}$ means taking the $\Sigma_n$-coinvariants.

\begin{definition}
A monoidal model category $\M$ is said to satisfy the \emph{strong commutative monoid axiom}  if, whenever $f: K\to L$ is a (trivial) cofibration, then so is $f^{\boxprod n}/{\Sigma_n}$, where $f^{\boxprod n}$ is the $n$-fold pushout product (which can be viewed as the unique morphism from the colimit $Q_n$ of a punctured $n$-dimensional cube to $L^{\otimes n}$), and the $\Sigma_n$-action is given by permuting the vertices of the cube.
\end{definition}

The following result says that, under suitable conditions, commutative Smith ideals and commutative monoid morphisms have equivalent homotopy theories.

\begin{corollary}\label{com-smith=map}
Suppose $\M$ is a cofibrantly generated stable monoidal model category that satisfies the strong commutative monoid axiom, the monoid axiom, and in which cofibrant $\arrowcompp$-algebras are also underlying cofibrant in $\arrowmppproj$ (this occurs, for example, if the monoidal unit is cofibrant).  Then there is a Quillen equivalence
\[\nicexy{\alg\bigl(\arrowcompp; \arrowmpp\bigr) \ar@<2pt>[r]^-{\mathrm{\coker}} & \alg\bigl(\arrowcomtensor; \arrowmtensor\bigr) \ar@<2pt>[l]^-{\mathrm{\ker}}}\]
in which $\Com$ is the commutative operad in $\M$.
\end{corollary}

\begin{proof}
First, \cite{white-commutative} (5.12 and 5.14) ensures that $\arrowmtensor$ and $\arrowmpp$ satisfy the strong commutative monoid axiom, and \cite{hovey-smith} (2.2 and 3.2) (also Theorems \ref{injective-model} and \ref{hovey-projective}) ensures that they satisfy the monoid axiom. Hence, by \cite{white-commutative} (3.2), $\alg\bigl(\arrowcompp; \arrowmpp\bigr)$ and $\alg\bigl(\arrowcomtensor; \arrowmtensor\bigr)$ carry transferred model structures.

For the commutative operad, it is proved in \cite{white-commutative} (3.6 and 5.14) that, with the strong commutative monoid axiom and a cofibrant monoidal unit, cofibrant $\arrowcompp$-algebras are also underlying cofibrant in $\arrowmppproj$.   So Theorem \ref{smith=map} applies. 
\end{proof}

\begin{example}[Commutative Smith Ideals in Symmetric Spectra]\label{com-spectra}
Example \ref{stable-examples} shows that the category of symmetric spectra with the positive flat stable model structure satisfies the hypotheses in Theorem \ref{smith=map}. It also satisfies the strong commutative monoid axiom \cite{white-commutative} (5.7) and the monoid axiom \cite{ss}. While the monoidal unit is not cofibrant, nevertheless, \cite{white-commutative} (5.15) shows that cofibrant commutative Smith ideals forget to cofibrant objects of $\arrowmpp$.  Therefore, Corollary \ref{com-smith=map} applies to the commutative operad $\Com$ in symmetric spectra with the positive flat stable model structure.
\end{example}

\begin{example}[Commutative Smith Ideals in Algebraic Settings]\label{com-algebra}
Let $R$ be a commutative ring containing the ring of rational numbers $\mathbb{Q}$. Corollary \ref{chain-zero} shows that the category of (bounded or unbounded) chain complexes of $R$-modules satisfies the conditions of Theorem \ref{smith=map}. They also satisfy the strong commutative monoid axiom and the monoid axiom \cite{white-commutative} (5.1). Hence, Corollary \ref{com-smith=map} applies, to give a homotopy theory of ideals of CDGAs. The same is true of the stable module category of $R = k[G]$ where $k$ is a field and $G$ is a finite group, using Corollary \ref{stmod-alloperad}. The result is a homotopy theory of ideals of commutative $R$-algebras.
\end{example}

\begin{example}[Commutative Smith Ideals in (Equivariant) Orthogonal/Symmetric Spectra]\label{com-spectra2}
Let $G$ be a compact Lie group. The positive flat stable model structure on $G$-equivariant orthogonal spectra satisfies the strong commutative monoid axiom \cite{white-commutative} (5.10), the monoid axiom \cite{white-localization} (Section 8), and has the property that cofibrant commutative Smith ideals forget to cofibrant objects of $\arrowmpp$ \cite{white-commutative} (5.15). 
The same is true for Hausmann's $G$-symmetric spectra built on either simplicial sets or topological spaces for a finite group $G$ by \cite{hausmann} (6.4, 6.16, 6.22), and for
Schwede's positive flat model structure for global equivariant homotopy theory (where commutative monoids are ultra-commutative ring spectra) \cite{schwede-global} (IV.3.28, V.4.1, V.4.3). 
Hence, Corollary \ref{com-smith=map} applies in all three settings.
\end{example}

Of course, taking $G$ trivial in Example \ref{com-spectra2}, one obtains that Corollary \ref{com-smith=map} applies to orthogonal spectra with the positive flat stable model structure \cite{white-localization} (Section 8).

\begin{remark} \label{remark:semi-com}
If, in Corollary \ref{com-smith=map}, $\M$ fails to satisfy the monoid axiom, then we still have semi-model structures on $\alg\bigl(\arrowcompp; \arrowmpp\bigr)$ and $\alg\bigl(\arrowcomtensor; \arrowmtensor\bigr)$ by \cite{white-commutative} (3.8). In this case, Theorem \ref{smith=map} still applies, as long as cofibrant $\arrowcompp$-algebras are also underlying cofibrant in $\arrowmppproj$ (e.g., if the monoidal unit is cofibrant, by \cite{white-commutative} (3.6)).
\end{remark}

\subsection{Smith Ideals for Sigma-Cofibrant Operads}
For a cofibrantly generated model category $\M$ and a small category $\D$, recall that the diagram category $\M^{\D}$ inherits a \emph{projective model structure} with weak equivalences and fibrations defined entrywise in $\M$ \cite{hirschhorn} (11.6.1).  We use this below when $\D = \Sigmacopc$ is the groupoid in Definition \ref{def:profiles}. In this case, the category $\M^{\D}$ is the category of $\fC$-colored symmetric sequences.

\begin{definition} \label{def:sigma-cof}
For a cofibrantly generated model category $\M$, a $\fC$-colored operad in $\M$ is said to be \emph{$\Sigmac$-cofibrant} if its underlying $\fC$-colored symmetric sequence is cofibrant.  If $\fC$ is the one-point set, then we say \emph{$\Sigma$-cofibrant} instead of $\Sigma_{\{*\}}$-cofibrant
\end{definition}

\begin{proposition}\label{sigmac-cofibrant-arrrowcat}
Suppose $\M$ is a cofibrantly generated model category, and $\D$ is a small category.  If $X \in \M^{\D}$ is cofibrant, then $L_1X \in (\arrowmppproj)^{\D}$ and $L_0X \in (\arrowmtensorinj)^{\D}$ are cofibrant.  In particular, if $\O$ is a $\Sigmac$-cofibrant $\fC$-colored operad in $\M$, then $\arrowopp = L_1\O$ is a $\Sigmac$-cofibrant $\fC$-colored operad in $\arrowmppproj$ and $\arrowotensor$ is a $\Sigmac$-cofibrant $\fC$-colored operad in $\arrowmtensorinj$.
\end{proposition}

\begin{proof}
The Quillen adjunction $L_1 : \M \adjoint \arrowmppproj : \Ev_1$ lifts to a Quillen adjunction of $\D$-diagram categories
\[\nicexy{\M^{\D} \ar@<2pt>[r]^-{L_1} & (\arrowmppproj)^{\D} \ar@<2pt>[l]^-{\Ev_1}}\]
by \cite{hirschhorn} (11.6.5(1)), and similarly for $(L_0,\Ev_0)$. If $X\in \M^{\D}$ is cofibrant, then $L_1X$ and $L_0X$ are cofibrant since $L_1$ and $L_0$ are left Quillen functors.
\end{proof}

The following result says that, under suitable conditions, for a $\Sigmac$-cofibrant $\fC$-colored operad $\O$, Smith $\O$-ideals and $\O$-algebra morphisms have equivalent homotopy theories.

\begin{corollary}\label{sigmacof-smith=map}
Suppose $\M$ is as in Theorem \ref{smith=map}, and $\O$ is a $\Sigmac$-cofibrant $\fC$-colored operad in $\M$.  Then $\algompp$ and $\algomtensor$ have transferred semi-model structures where cofibrant $\arrowopp$-algebras are also underlying cofibrant in $(\arrowmppproj)^{\fC}$. Hence, there is a Quillen equivalence
\[\nicexy{\algompp \ar@<3pt>[r]^-{\coker} & \algomtensor. \ar@<1pt>[l]^-{\ker}}\]
\end{corollary}

\begin{proof}
The arrow categories $\arrowmppproj$ and $\arrowmtensorinj$ are cofibrantly generated monoidal model categories by Theorems \ref{injective-model} and \ref{hovey-projective}.  By Proposition \ref{sigmac-cofibrant-arrrowcat}, the $\fC$-colored operads $\arrowopp$ in $\arrowmppproj$ and $\arrowotensor$ in $\arrowmtensorinj$ are $\Sigmac$-cofibrant.  Theorem 6.3.1 in \cite{white-yau}, applied to $\arrowmppproj$ and $\arrowmtensorinj$, now gives the transferred semi-model structures and says that every cofibrant $\arrowopp$-algebra is underlying cofibrant in ($\arrowmppproj)^{\fC}$.  So Theorem \ref{smith=map} applies.
\end{proof}

The following provides one source of applications of Corollary \ref{sigmacof-smith=map}, and answers a question Pavel Safranov asked the first author. This result generalizes \cite{white-commutative} (5.1) and \cite{white-yau} (8.1), as it applies in particular to fields of characteristic zero.

\begin{corollary}\label{chain-zero}
Suppose $R$ is a commutative ring with unit and $\M$ is the category of bounded or unbounded chain complexes of $R$-modules, with the projective model structure. The following are equivalent:
\begin{enumerate}
\item $R$ is a semi-simple ring containing the rational numbers $\mathbb{Q}$.
\item Every symmetric sequence is projectively cofibrant.
\end{enumerate}
In particular, for such rings $R$, every $\fC$-colored operad in $\M$ is $\Sigmac$-cofibrant, so Corollary \ref{sigmacof-smith=map} is applicable for all colored operads in $\M$. If $R$ contains $\mathbb{Q}$ (but is not necessarily semi-simple) then every entrywise cofibrant $\fC$-colored operad in $\M$ is $\Sigmac$-cofibrant and admissible.
\end{corollary}

\begin{proof}
Assume (1). Maschke's Theorem \cite{milies-book} (3.4.7) guarantees that each group ring $R[\Sigma_n]$ is semi-simple (since $\frac{1}{n!}$ exists in $R$, making $n!$ invertible). This means every module $M$ over $R[\Sigma_n]$ is projective. In particular, $M$ is a direct summand of a module induced from the trivial subgroup, and has a free $\Sigma_n$-action. Hence, (2) follows.

Conversely, if (2) is true, then it implies that, for every $n$, every module in $R[\Sigma_n]$ is projective. This means each $R[\Sigma_n]$ is a semi-simple ring. By \cite{milies-book} (3.4.7), this implies that $R$ is semi-simple and $n!$ is invertible in $R$ for every $n$. It follows that $\mathbb{Q}$ is contained in $R$.

For such $R$, the projective model structure on (bounded or unbounded) chain complexes of $R$-modules has every object cofibrant (so, automatically, cofibrant operad-algebras forget to cofibrant chain complexes). Hence, any $\fC$-colored operad is entrywise cofibrant, and hence $\Sigmac$-cofibrant. Furthermore, Theorem \ref{spade} implies that all operads are admissible, since every $X\in \M^{\Sigma_n^{op}}$ is $\Sigma_n$-projectively cofibrant. 

If $R$ contains $\mathbb{Q}$ but is not semi-simple, then there can be non-projective $R$-modules, but the argument of \cite{milies-book} (3.4.7) shows that an $R[\Sigma_n]$-module that is projective as an $R$-module is projective as a $R[\Sigma_n]$-module. It follows that Corollary \ref{sigmacof-smith=map} holds for entrywise cofibrant operads, including the operad $Com$. Indeed, all operads are admissible thanks to Theorem \ref{spade}, since for any trivial cofibration $f$ and any $X \in \M^{\Sigma_n^{op}}$, maps of the form $X\otimes_{\Sigma_n} f^{\boxprod n}$ are trivial $h$-cofibrations and this class of morphisms is closed under pushout and transfinite composition \cite{white-localization} (Section 8).
\end{proof}

\begin{example}\label{applicable-operads}
Suppose $\M$ is as in Theorem \ref{smith=map}, that is, cofibrantly generated, stable, monoidal, and with (co)domains of $I\cup J$ small. Many examples of such $\M$ are provided in Example \ref{stable-examples}, Example \ref{stable-examples2}, and in our papers \cite{white-commutative, white-yau, gutierrez-white-equivariant, white-yau3, white-yau-arrowcat, white-yau2, white-yau-co, hovey-white, white-localization}.  Here are some examples of $\Sigma$-cofibrant operads, for which Corollary \ref{sigmacof-smith=map} is applicable.
\begin{description}
\item[Smith Ideals] The associative operad $\As$, which has $\As(n) = \coprod_{\Sigma_n} \tensorunit$ as the $n$th entry and which has monoids as algebras, is $\Sigma$-cofibrant.  In this case, Corollary \ref{sigmacof-smith=map} is Hovey's Corollary 4.4 (1) in \cite{hovey-smith}.
\item[Smith $A_{\infty}$-Ideals] Any $A_{\infty}$-operad, defined as a $\Sigma$-cofibrant resolution of $\As$, is $\Sigma$-cofibrant.  In this case, Corollary \ref{sigmacof-smith=map} says that Smith $A_{\infty}$-ideals and $A_{\infty}$-algebra morphisms have equivalent homotopy theories. For instance, one can take the standard differential graded $A_{\infty}$-operad \cite{markl} and, for symmetric spectra, the Stasheff associahedra operad \cite{stasheff}.
\item[Smith $E_{\infty}$-Ideals] Any $E_{\infty}$-operad, defined as a $\Sigma$-cofibrant resolution of the commutative operad $\Com$, is $\Sigma$-cofibrant.  In this case, Corollary \ref{sigmacof-smith=map} says that Smith $E_{\infty}$-ideals and $E_{\infty}$-algebra morphisms have equivalent homotopy theories.  For example, for symmetric spectra, one can take the Barratt-Eccles $E_{\infty}$-operad $E\Sigma_*$ \cite{be}.  An elementary discussion of the Barratt-Eccles operad is in \cite{cerberusIII} (Section 11.4).
\item[Smith $E_n$-Ideals] For each $n \geq 1$, the little $n$-cubes operad $\C_n$ \cite{boardman-vogt,may72} is $\Sigma$-cofibrant and is an $E_n$-operad by definition \cite{fresse-gt} (4.1.13). In this case, with $\M$ being symmetric spectra with the positive (flat) stable model structure, Corollary \ref{sigmacof-smith=map} says that Smith $\C_n$-ideals and $\C_n$-algebra morphisms have equivalent homotopy theories.  One may also use other $\Sigma$-cofibrant $E_n$-operads \cite{fie}, such as the Fulton-MacPherson operad (\cite{getzler-jones} and \cite{fresse-gt} (4.3)), which is actually a cofibrant $E_n$-operad.  An elementary discussion of a categorical $E_n$-operad is in \cite{cerberusIII} (Chapter 13).
\end{description}
\end{example}

\begin{example}\label{stable-examples2}
The power of restricting attention to the class of $\Sigmac$-cofibrant colored operads is that Theorem \ref{smith=map} holds for a larger class of model categories. In particular, the following model categories satisfy the conditions of Theorem \ref{smith=map} for the class of $\Sigmac$-cofibrant colored operads, as do all examples listed in Subsection \ref{subsec:comm}.
\begin{enumerate}
\item $S$-modules with the model structure from \cite{ekmm}.
\item The projective, injective, positive, or positive flat stable model structures \cite{white-localization} (8.3, 8.5) on symmetric spectra, $G$-equivariant orthogonal spectra (for a compact Lie group $G$), and motivic symmetric spectra.
\item Mandell's model structure on $G$-equivariant symmetric spectra built on simplicial sets or topological spaces, where $G$ is a finite group in the former case and a compact Lie group in the latter case \cite{mandell-equivariant}.
\item Model structures for (equivariant) stable homotopy theory based on Lydakis's theory of enriched functors \cite{dro}. For example, this includes the model category of $G$-enriched functors from finite $G$-simplicial sets to $G$-simplicial sets, where $G$ is a finite group, from \cite{dro} (Theorem 2).
\item Any model structure $\M$ on symmetric spectra built on ($\cat C, G)$ where $\cat C$ is a model category and $G$ is an endofunctor, as long as $\M$ is an operadically cofibrantly generated, monoidal, stable model structure. For example, taking $\cat C$ to be the canonical model structure on small categories, and using the suspension discussed in \cite{white-yau2} (Section 13), one obtains by \cite{hovey-spectra} (7.3) a combinatorial, stable, monoidal model structure on symmetric spectra of small categories with applications to Goodwillie calculus. Using \cite{dmitri} (Section 2) one may obtain positive and positive flat variants. Another example is taking $\cat C$ to be the $I$-spaces or $J$-spaces of Sagave and Schlichtkrull, and building projective, positive, or positive flat spectra on them as in \cite{dmitri} (Section 2).
\item The projective model structure on bounded or unbounded chain complexes over a commutative ring $R$ \cite{white-yau2} (Section 11).
\item The stable module category of $k[G]$ where $G$ is a finite group and $k$ is a principal ideal domain \cite{white-yau2} (Section 12).
\end{enumerate}
All of these examples are stable monoidal model categories, so Corollary \ref{sigmacof-smith=map} applies, once the requisite smallness hypothesis for the generating (trivial) cofibrations is checked. Symmetric spectra, motivic symmetric spectra, examples (6) and (7), and Mandell's model (3) of $G$-equivariant symmetric spectra built on simplicial sets are all combinatorial, as is the model structure on enriched functors (4) in simplicial contexts. Symmetric spectra as in (5) are combinatorial if $\cat C$ is combinatorial. $S$-modules, $G$-equivariant orthogonal spectra, Mandell's model (3) in topological contexts, and symmetric spectra built on topological spaces (another example of (5)) are operadically cofibrantly generated just as in Example \ref{example:top-operadically-cof-gen}, since they are built from compactly generated spaces.  We recall that spaces are small relative to inclusions, and the morphisms in $(\O \circ (I\cup J))$-cell are inclusions \cite{white-yau2} (5.10).

\end{example}

\section{Smith Ideals for Entrywise Cofibrant Operads}\label{sec:entywise}

In this section we apply Theorem \ref{smith=map} to operads that are not necessarily $\Sigmac$-cofibrant.  To do that, we need to redistribute some of the cofibrancy assumptions---that cofibrant Smith $\O$-ideals are underlying cofibrant in the arrow category---from the colored operad to the underlying category.  We will show in Theorem \ref{underlying-cofibrant} that Theorem \ref{smith=map} is applicable to all entrywise cofibrant operads, provided that $\M$ satisfies the cofibrancy condition $(\heartsuit)$ below.  This implies that, over the stable module category \cite{hovey} (2.2), Theorem \ref{smith=map} is always applicable.

\subsection{Cofibrancy Assumptions}

\begin{definition}\label{heart-conditions}
Suppose $\M$ is a cofibrantly generated monoidal model category.  Define the following conditions in $\M$.
\begin{description}
\item[$(\heartsuit)$] For each $n \geq 1$ and each morphism $f \in \Msigmanop$ that is an underlying cofibration between cofibrant objects in $\M$, the function
\[f \boxprodover{\Sigma_n} (-) : \M^{\Sigma_n} \to \M\]
takes each morphism in $\M^{\Sigma_n}$ that is an underlying cofibration in $\M$ to a cofibration in $\M$.  More explicitly, this condition asks that, for each morphism $g \in \M^{\Sigma_n}$ that is an underlying cofibration in $\M$, the morphism
\[f \boxprodover{\Sigma_n} g = (f \boxprod g)/\Sigma_n\]
is a cofibration in $\M$.
\item[$\clubcof$] For each $n \geq 1$ and each object $X \in \Msigmanop$ that is underlying cofibrant in $\M$, the function
\[X \tensorover{\Sigma_n} (-)^{\boxprod n} : \M \to \M\]
preserves cofibrations.
\item[$\clubtcof$] For each $n \geq 1$ and each object $X \in \Msigmanop$ that is underlying cofibrant in $\M$, the function
\[X \tensorover{\Sigma_n} (-)^{\boxprod n} : \M \to \M\]
preserves trivial cofibrations.
\end{description}
\end{definition}

\begin{remark}
The condition $(\heartsuit)$ implies $\clubcof$, since $(\varnothing \to X) \boxprod (-) = X \otimes (-)$.  The condition $\clubcof$ was introduced in \cite{white-yau} (6.2.1), where the authors proved that, if $\M$ satisfies $\clubcof$ and $\clubtcof$, then there exist transferred semi-model structures on algebras over entrywise cofibrant (but not necessarily $\Sigmac$-cofibrant) colored operads. It is, therefore, no surprise that we consider $\clubcof$ and its variant $(\heartsuit)$ here in order to use Theorem \ref{smith=map} for operads that are not necessarily $\Sigmac$-cofibrant. Of course, $(\spadesuit)$ implies $\clubtcof$, so $\clubtcof$ holds in all the model categories in Example \ref{spade-examples}. 
\end{remark}

\begin{proposition}
The condition $(\heartsuit)$ holds in the following categories:
\begin{enumerate}
\item Simplicial sets with either the Quillen model structure or the Joyal model structure \cite{lurie}, where cofibrations are the monomorphisms;
\item Bounded or unbounded chain complexes over a field $k$ of characteristic zero, where cofibrations are degreewise monomorphisms \cite{hovey} (2.3.9) since every monomorphism of $k$-modules splits, and every chain complex is cofibrant (see Corollary \ref{chain-zero}); 
\item Small categories with the folk model structure where cofibrations are injective on objects \cite{rezk}; 
\item The stable module category of $k[G]$-modules with the characteristic of $k$ dividing the order of $G$, where cofibrations are injections \cite{hovey} (2.2.12); and
\item The injective model structure on symmetric spectra, $G$-equivariant symmetric spectra, and motivic symmetric spectra, where the cofibrations are the monomorphisms \cite{hovey-spectra}.
\end{enumerate}
\end{proposition}

\begin{proof}
For simplicial sets with either model structure, a cofibration is precisely an injection, and the pushout product of two injections is again an injection.  Dividing an injection by a $\Sigma_n$-action is still an injection.  The other cases are proved similarly.
\end{proof}

\begin{proposition}
If $(\heartsuit)$ holds in $\M$, then it also holds in any left Bousfield localization of $\M$.
\end{proposition}

\begin{proof}
The condition $(\heartsuit)$ only refers to cofibrations, which remain the same in any left Bousfield localization.
\end{proof}

The next observation is the key that connects the cofibrancy condition $(\heartsuit)$ in $\M$ to the arrow category.

\begin{theorem}\label{arrowm-club}
Suppose $\M$ is a cofibrantly generated monoidal model category satisfying $(\heartsuit)$.  Then the arrow category $\arrowmppproj$ satisfies $\clubcof$.
\end{theorem}

\begin{proof}
Suppose $f_X : X_0 \to X_1$ is an object in $(\arrowmppproj)^{\Sigmanop}$ that is underlying cofibrant in $\arrowmppproj$.  This means that $f_X$ is a morphism in $\Msigmanop$ that is an underlying cofibration between cofibrant objects in $\M$.  The condition $\clubcof$ for $\arrowmppproj$ asks that the function
\[f_X \boxprodover{\Sigma_n} (-)^{\boxprod_2 n} : \arrowmppproj \to \arrowmppproj\]
preserve cofibrations, where $\boxprod$ and $\boxprod_2$ are the pushout products in $\M$ and $\arrowmpp$, respectively.  When $n=1$ the condition $\clubcof$ for $\arrowmppproj$ is a special case of the pushout product axiom in $\arrowmppproj$, which is true by Theorem A in \cite{white-yau-arrowcat}.

Next suppose $n \geq 2$ and $\alpha : f_V \to f_W$ is a morphism in $\arrowm$ as in \eqref{alphavw}.  The iterated pushout product $\alpha^{\boxprod_2 n} \in (\arrowmpp)^{\Sigma_n}$ is the commutative square
\begin{equation}\label{alphaboxn}
\nicexy@R-.3cm{Z \ar[d]_-{\zeta_0} \ar[r]^-{\zeta_1} & Y_1 \ar[d]^-{f_W^{\boxprod n}}\\
Y_0 \ar[r]^-{\alpha_1^{\boxprod n}} & W_1^{\otimes n}}
\end{equation}
in $\M^{\Sigma_n}$ for some object $Z$ with $\zeta_1 = \Ev_0(\alpha^{\boxprod_2 n})$.  Note that $\zeta_1$ is not an iterated pushout product because $\Ev_0$ and $\boxprod_2$ do not commute.  Applying $f_X \boxprod_{\Sigma_n} (-)$, the morphism $f_X \boxprod_{\Sigma_n} \alpha^{\boxprod_2 n}$ is the commutative square
\begin{equation}\label{falpha-boxn}
\scalebox{.9}{$
\nicexy{\Bigl[(X_1 \otimes Z) \coprodover{X_0 \otimes Z} (X_0 \otimes Y_0)\Bigr]_{\Sigma_n} \ar[d]_-{f_X \boxprodover{\Sigma_n} \zeta_0} \ar[r]^-{\varphi} 
& \Bigl[(X_1 \otimes Y_1) \coprodover{X_0 \otimes Y_1} (X_0 \otimes W_1^{\otimes n})\Bigr]_{\Sigma_n} \ar[d]^-{f_X \boxprodover{\Sigma_n} f_W^{\boxprod n}}\\
(X_1 \otimes Y_0)_{\Sigma_n} \ar[r]^-{(X_1 \otimes\, \alpha_1^{\boxprod n})_{\Sigma_n}}
& (X_1 \otimes W_1^{\otimes n})_{\Sigma_n}}$}
\end{equation}
in $\M$.  Suppose $\alpha$ is a cofibration in $\arrowmppproj$.  This means that the morphism $\alpha_0 : V_0 \to W_0$ and the pushout corner morphism $\alpha_1 \pcorner f_W : V_1 \coprod_{V_0} W_0 \to W_1$ are cofibrations in $\M$.  We must show that $f_X \boxprod_{\Sigma_n} \alpha^{\boxprod_2 n}$ is a cofibration in $\arrowmppproj$.  In other words, we must show that in \eqref{falpha-boxn}:
\begin{enumerate}
\item $\varphi = \Ev_0\bigl(f_X \boxprod_{\Sigma_n} \alpha^{\boxprod_2 n}\bigr)$ is a cofibration in $\M$.
\item The pushout corner morphism of $f_X \boxprod_{\Sigma_n} \alpha^{\boxprod_2 n}$ is a cofibration in $\M$. 
\end{enumerate}
We will prove (1) and (2) in Lemmas \ref{lemma-heart1} and \ref{lemma-heart2}, respectively.
\end{proof}

\begin{lemma}\label{lemma-heart1}
The morphism $\varphi$ in \eqref{falpha-boxn} is a cofibration in $\M$.
\end{lemma}

\begin{proof}
Taking $\Sigma_n$-coinvariants and taking pushouts commute by the commutation of colimits.  So $\varphi$ is also the induced morphism from the pushout of the top row to the pushout of the bottom row in the commutative diagram
\begin{equation}\label{varphi-is-pushout}
\nicexy@C+.4cm{(X_1 \otimes Z)_{\Sigma_n} \ar[d]_-{(X_1 \otimes\, \zeta_1)_{\Sigma_n}} & (X_0 \otimes Z)_{\Sigma_n} \ar[l]_-{(f_X \otimes Z)_{\Sigma_n}} \ar[d]_{(X_0 \,\otimes\, \zeta_1)_{\Sigma_n}} \ar[r]^-{(X_0 \,\otimes\, \zeta_0)_{\Sigma_n}} & (X_0 \otimes Y_0)_{\Sigma_n} \ar[d]^-{(X_0 \,\otimes\, \alpha_1^{\boxprod n})_{\Sigma_n}}\\
(X_1 \otimes Y_1)_{\Sigma_n} & (X_0 \otimes Y_1)_{\Sigma_n} \ar[l]^-{(f_X \otimes Y_1)_{\Sigma_n}} \ar[r]_-{(X_0 \,\otimes\, f_W^{\boxprod n})_{\Sigma_n}} & (X_0 \otimes W_1^{\otimes n})_{\Sigma_n}}
\end{equation}
in $\M$.  Here the left square is commutative by definition, and the right square is $X_0 \otimes_{\Sigma_n} (-)$ applied to $\alpha^{\boxprod_2 n}$ in \eqref{alphaboxn}.

We consider the Reedy category $\sD$ with three objects $\{-1,0,1\}$, a morphism $0 \to -1$ that lowers the degree, a morphism $0 \to 1$ that raises the degree, and no other non-identity morphisms.  Using the Quillen adjunction \cite{hovey} (proof of 5.2.6)
\[\nicexy@C+.5cm{\M^{\sD} \ar@<2pt>[r]^-{\colim} & \M \ar@<2pt>[l]^-{\mathrm{constant}}}\]
to show that $\varphi$ is a cofibration in $\M$, it is enough to show that \eqref{varphi-is-pushout} is a Reedy cofibration in $\M^{\sD}$.  So we must show that in \eqref{varphi-is-pushout}:
\begin{enumerate}
\item The left and the middle vertical arrows are cofibrations in $\M$.
\item The pushout corner morphism of the right square is a cofibration in $\M$.
\end{enumerate}

The objects $X_0$ and $X_1$ in $\Msigmanop$ are cofibrant in $\M$.  The morphism $\zeta_1 = \Ev_0(\alpha^{\boxprod_2 n}) \in \M^{\Sigma_n}$ is an underlying cofibration in $\M$.  Indeed, since $\alpha \in \arrowmppproj$ is a cofibration, so is the iterated pushout product $\alpha^{\boxprod_2 n}$ by the pushout product axiom \cite{white-yau-arrowcat}.  In particular, $\Ev_0(\alpha^{\boxprod_2 n})$ is a cofibration in $\M$.  The condition $(\heartsuit)$ in $\M$ (for the morphism $\varnothing \to X_i$) now implies that the left and the middle vertical morphisms $X_i \otimes_{\Sigma_n} \zeta_1$ in \eqref{varphi-is-pushout} are cofibrations in $\M$.

Finally, since $X_0 \in \Msigmanop$ is cofibrant in $\M$ and since the pushout corner morphism of $\alpha^{\boxprod_2 n} \in (\arrowmppproj)^{\Sigma_n}$ is a cofibration in $\M$, the condition $(\heartsuit)$ in $\M$ again implies the pushout corner morphism of the right square $X_0 \otimes_{\Sigma_n} \alpha^{\boxprod_2 n}$ in \eqref{varphi-is-pushout} is a cofibration in $\M$. 
\end{proof}

\begin{lemma}\label{lemma-heart2}
The pushout corner morphism of $f_X \boxprod_{\Sigma_n} \alpha^{\boxprod_2 n}$ in \eqref{falpha-boxn} is a cofibration in $\M$. 
\end{lemma}

\begin{proof}
The pushout corner morphism of $f_X \boxprod_{\Sigma_n} \alpha^{\boxprod_2 n}$ is the morphism $f_X \boxprod_{\Sigma_n} (\alpha_1^{\boxprod n} \pcorner f_W^{\boxprod n})$.  This is the $\Sigma_n$-coinvariants of the pushout product in the diagram
\[\scalebox{.8}{$
\nicexy@C-.4cm@R-.5cm{X_0 \otimes \bigl(Y_0 \coprodover{Z} Y_1\bigr) \ar@{}[dr]|-{\mathrm{pushout}} \ar[d]_-{f_X \otimes\, \Id} \ar[r]^-{\Id \otimes (\alpha_1^{\boxprod n} \pcorner f_W^{\boxprod n})} & X_0 \otimes W_1^{\otimes n} 
\ar[d] \ar@/^3pc/[ddr]^-{f_X \otimes\, \Id} &\\
X_1 \otimes \bigl(Y_0 \coprodover{Z} Y_1\bigr) \ar[r] \ar@/_1.5pc/[drr]_-{\Id \otimes  (\alpha_1^{\boxprod n} \pcorner f_W^{\boxprod n})} & \Bigl[X_1 \otimes \bigl(Y_0 \coprodover{Z} Y_1\bigr)\Bigr] \coprodover{\bigl[X_0 \otimes (Y_0 \coprodover{Z}Y_1)\bigr]} (X_0 \otimes W_1^{\otimes n}) 
\ar@{}[dr]^(.35){}="a"^(.9){}="b" \ar "a";"b" |-{f_X \boxprod (\alpha_1^{\boxprod n} \pcorner f_W^{\boxprod n})} &\\ && X_1 \otimes W_1^{\otimes n}}$}\]
in $\M^{\Sigma_n}$ with $\alpha_1^{\boxprod n} \pcorner f_W^{\boxprod n}$ the pushout corner morphism of $\alpha^{\boxprod_2 n} \in (\arrowmppproj)^{\Sigma_n}$ in \eqref{alphaboxn}.  Since $\alpha^{\boxprod_2 n}$ is a cofibration in $\arrowmppproj$, its pushout corner morphism $\alpha_1^{\boxprod n} \pcorner f_W^{\boxprod n}$ is a cofibration in $\M$.  So the condition $(\heartsuit)$ in $\M$ implies that $f_X \boxprod_{\Sigma_n} (\alpha_1^{\boxprod n} \pcorner f_W^{\boxprod n})$ is a cofibration in $\M$.
\end{proof}

\subsection{Underlying Cofibrancy of Cofibrant Smith Ideals for Entrywise Cofibrant Operads}

\begin{theorem}\label{underlying-cofibrant}
Suppose $\M$ is a cofibrantly generated monoidal model category satisfying $(\heartsuit)$ and $\clubtcof$, in which the domains and the codomains of the generating (trivial) cofibrations are small with respect to the entire category.  Suppose $\O$ is an entrywise cofibrant $\fC$-colored operad in $\M$.  Then $\algompp$ and $\algomtensor$ admit transferred semi-model structures, and cofibrant Smith $\O$-ideals are underlying cofibrant in $(\arrowmppproj)^{\fC}$.  In particular, if $\M$ is also stable, then there is a Quillen equivalence
\[\nicexy{\algompp \ar@<3pt>[r]^-{\coker} & \algomtensor. \ar@<1pt>[l]^-{\ker}}\]
\end{theorem}

\begin{proof}
If $\O$ is entrywise cofibrant in $\M$, then $\arrowopp = L_1 \O$ is entrywise cofibrant in $\arrowmppproj$, and $\arrowotensor = L_0 \O$ is entrywise cofibrant in $\arrowmtensorinj$ by Proposition \ref{sigmac-cofibrant-arrrowcat}.  Furthermore, because $\M$ satisfies $\clubtcof$, so does $\arrowmtensorinj$, by the exact same proof as in Theorem \ref{spade-arrow-tensor} (but now $X_0$ and $X_1$ are cofibrant in $\M$, and we appeal to $\clubtcof$ in $\M$ instead of $(\spadesuit)$). Thus, we have transferred semi-model structures
\begin{itemize}
\item $\algomtensor$ by \cite{white-yau} (6.2.3) applied to $\arrowmtensorinj$ and
\item $\algompp$ by \cite{white-yau} (6.2.3) applied to the colored operad $\O^s$ in $\M$ in Corollary \ref{smitho-m}.
\end{itemize} 

Using Theorem \ref{smith=map}, it is enough to prove the assertion that cofibrant Smith $\O$-ideals are underlying cofibrant in $(\arrowmppproj)^{\fC}$.  Writing $\varnothing^{\arrowopp}$ for the initial $\arrowopp$-algebra, first we claim that $\varnothing^{\arrowopp}$ is underlying cofibrant in $(\arrowmppproj)^{\fC}$.  Indeed, for each color $d \in \fC$, the $d$-colored entry of the initial $\arrowopp$-algebra is the object 
\[\varnothing^{\arrowopp}_d = \arrowopp\dnothing = \left(\varnothing^{\M} \to \O\dnothing\right)\]
in $\arrowmppproj$, where $\varnothing^{\M}$ is the initial object in $\M$ and the symbol $\varnothing$ in $\dnothing$ is the empty $\fC$-profile.  Since $\O$ is assumed entrywise cofibrant, it follows that each entry of the initial $\arrowopp$-algebra $\varnothing^{\arrowopp}$ is underlying cofibrant in $\arrowmppproj$.  Indeed, the pushout corner morphism of 
\[\begin{tikzcd}
\varnothing^{\M} \ar{d} \ar{r} & \varnothing^{\M} \ar{d}\\
\varnothing^{\M} \ar{r} & \O\dnothing
\end{tikzcd}\]
is the cofibration $\varnothing^{\M} \to \O\dnothing$ in $\M$, so by Theorem \ref{hovey-projective} (1) $\varnothing^{\arrowopp}_d$ is cofibrant in $\arrowmppproj$.

By Proposition \ref{cgoverarrowm}, the semi-model structure on $\algompp$ is right-induced by the forgetful functor $U$ to $(\arrowmppproj)^{\fC}$ and is cofibrantly generated by $\arrowopp \circ (L_0I \cup L_1I)_c$ and $\arrowopp \circ (L_0J \cup L_1J)_c$ for $c \in \fC$, where $I$ and $J$ are the generating (trivial) cofibrations in $\M$. Suppose $A$ is a cofibrant $\arrowopp$-algebra.  We must show that $A$ is underlying cofibrant in $(\arrowmppproj)^{\fC}$.  By \cite{hirschhorn} (11.2.2), the cofibrant $\arrowopp$-algebra $A$ is the retract of the colimit of a transfinite composition, starting with $\varnothing^{\arrowopp}$, of pushouts of morphisms in $\arrowopp \circ (L_0I \cup L_1I)_c$ for $c \in \fC$.  Since $\varnothing^{\arrowopp}$ is underlying cofibrant in $\arrowmppproj$ and since the class of cofibrations in a model category, such as $(\arrowmppproj)^{\fC}$, is closed under transfinite compositions \cite{hirschhorn} (10.3.4), the following Lemma will finish the proof.
\end{proof}

The proof of Lemma \ref{pushout-undelrying-cof} below uses the next definition from \cite{white-yau} (4.3.5).

\begin{definition}[$\O_A$ for $\O$-algebras]
\label{oaalgebra}
For a $\fC$-colored operad $\O$ in $\M$ and $A \in \algom$, define the $\fC$-colored symmetric sequence $\O_A$ as follows.  For $d \in \fC$ and orbit $[\uc] \in \Sigmac$, define the component
\[\O_A\smallbinom{d}{[\uc]} \in \M^{\Sigmaop_{[\uc]} \times \{d\}}\]
as the reflexive coequalizer of the following diagram in $\M^{\Sigmaop_{[\uc]} \times \{d\}}$.
\[\begin{tikzcd}
\coprod\limits_{[\ua] \in \Sigmac} \O\smallbinom{d}{[\ua],\,[\uc]}
\tensorover{\Sigma_{[\ua]}} (\O \circ A)_{[\ua]}
\ar{r}{d_1} \ar[shift right=2]{r}[swap]{d_0} &
\coprod\limits_{[\ua] \in \Sigmac} \O\smallbinom{d}{[\ua],\,[\uc]}
\tensorover{\Sigma_{[\ua]}} A_{[\ua]} 
\ar[bend right=15, shorten <=-1ex, shorten >=-1ex]{l}[swap]{s}
\end{tikzcd}\]
The three arrows in this diagram are as follows:
\begin{itemize}
\item
$d_0$ is induced by the composition of $\O$.
\item
$d_1$ is induced by the $\O$-algebra structure on $A$.
\item
The common section $s$ is induced by the unit $A \to \O \circ A$.
\end{itemize}
\end{definition}

\begin{lemma}\label{pushout-undelrying-cof}
Under the hypotheses of Theorem \ref{underlying-cofibrant}, suppose $\alpha : f \to g$ is a morphism in $(L_0I \cup L_1I)_c$ for some color $c \in \fC$, and 
\[\nicexy@R-.3cm{\arrowopp \circ f \ar[r] \ar[d]_-{\arrowopp \circ \alpha} & B_0 \ar[d]^-{j}\\
\arrowopp \circ g \ar[r] & B_{\infty}}\]
is a pushout in $\algompp$ with $B_0$ cofibrant and $UB_0 \in (\arrowmppproj)^{\fC}$ cofibrant.  Then $Uj$ is a cofibration in $(\arrowmppproj)^{\fC}$.  In particular, $B_{\infty}$ is also cofibrant and $UB_{\infty} \in (\arrowmppproj)^{\fC}$ is cofibrant.
\end{lemma}

\begin{proof}
By the filtration in \cite{white-yau} (4.3.16) and the fact that cofibrations are closed under pushouts, to show that $Uj \in (\arrowmppproj)^{\fC}$ is a cofibration, it is enough to show that, for each $n \geq 1$ and each color $d \in \fC$, the morphism 
\begin{equation}\label{arrowosuba}
\arrowopp_{B_0}\dnc \boxprodover{\Sigma_n} \alpha^{\boxprod_2 n}
\end{equation} 
in $\arrowmppproj$ is a cofibration, where $nc = (c,\ldots,c)$ is the $\fC$-profile with $n$ copies of the color $c$.  The object $\arrowopp_{B_0}$ is as in Definition \ref{oaalgebra} for $\arrowopp$ and $B_0$, and $\alpha^{\boxprod_2 n}$ is the $n$-fold pushout product of $\alpha$.  Recall that $\arrowmppproj$ satisfies $\clubcof$ by Theorem \ref{arrowm-club} and that $\arrowopp$ is entrywise cofibrant in $\arrowmppproj$ because $\O$ is entrywise cofibrant in $\M$.  The cofibrancy of $B_0 \in \algompp$ and \cite{white-yau} (6.2.4) applied to $\arrowopp$ now imply that $\arrowopp_{B_0}$ is entrywise cofibrant in $\arrowmppproj$.  By the condition $\clubcof$ in $\arrowmppproj$ once again, we can conclude that the morphism \eqref{arrowosuba} is a cofibration because $\alpha$ is a cofibration in $\arrowmppproj$.
\end{proof}

\begin{corollary}\label{stmod-alloperad}
Suppose $\M$ is the stable module category of $k[G]$-modules for some field $k$ whose characteristic divides the order of $G$.  Then for each $\fC$-colored operad $\O$ in $\M$, there is a Quillen equivalence
\[\nicexy{\algompp \ar@<3pt>[r]^-{\coker} & \algomtensor. \ar@<1pt>[l]^-{\ker}}\]
\end{corollary}

\begin{proof}
The stable module category is a stable model category that satisfies the hypotheses of Theorem \ref{underlying-cofibrant} in which every object is cofibrant \cite{hovey} (2.2.12), \cite{white-yau2} (Section 12).
\end{proof}

There are several more examples where Theorem \ref{smith=map} likely applies to all entrywise cofibrant operads, but where $(\heartsuit)$ has not been checked. For example, the positive flat stable model structure on symmetric spectra built on compactly generated spaces have the property that, for any entrywise cofibrant colored operad $\O$, cofibrant $\O$-algebras forget to cofibrant spectra \cite{dmitri} (Section 2), but the authors do not know a reference proving the same for $\arrowmppproj$. 

\begin{conjecture}
The positive flat stable model structure on symmetric spectra built on compactly generated spaces satisfies the conclusion of Theorem \ref{underlying-cofibrant}.
\end{conjecture}

Similarly, by analogy with the positive flat model structure on symmetric spectra, one would expect that the positive flat model structure on $G$-equivariant orthogonal spectra would satisfy this property. 

\begin{conjecture}
If $\M = G\Sp_O$ is the positive flat stable model structure on $G$-equivariant orthogonal spectra, then it  satisfies the property that, if $\O$ is an entrywise cofibrant $\fC$-colored operad and $A$ is a cofibrant $\O$-algebra then $UA$ is cofibrant in $\M^{\fC}$. Furthermore, $\M$ satisfies the conclusion of Theorem \ref{underlying-cofibrant}, for any compact Lie group $G$. 
\end{conjecture}

Recent work of Hill, Hopkins, and Ravenel has illustrated that the positive (flat) model structure on $G\Sp_O$ is not quite right. One also needs an equifibrancy condition, also known as completeness. There is a positive complete model structure on $G\Sp_O$, and it satisfies the commutative monoid axiom \cite{gutierrez-white-equivariant} (Section 5). However, the authors do not know if a positive, complete, flat variant has been worked out.

\begin{problem}
Let $G$ be a compact Lie group. 
\begin{enumerate}
\item Work out a positive complete flat stable model structure on $G\Sp_O$.
\item Prove that it satisfies the condition that all colored operads are admissible.
\item Prove that cofibrant operad-algebras forget to cofibrant underlying objects. 
\item Prove that this model structure satisfies the conclusion of Theorem \ref{underlying-cofibrant}. 
\end{enumerate}
\end{problem}

In a related vein, we have the following problem.

\begin{problem}
Let $\M_s$ (resp. $\M_s^+$) denote Schwede's global positive (flat) model structure \cite{schwede-global} and let $\M_h$ (resp. $\M_h^+$) denote Hausmann's positive (flat) model structure for $G$-symmetric spectra \cite{hausmann}. 
\begin{enumerate}
\item Prove that all colored operads are admissible in $\M_s$, $\M_s^+$, $\M_h$, and $\M_h^+$.
\item Prove that, if $\O$ is entrywise cofibrant, then cofibrant $\O$-algebras forget to underlying cofibrant objects in $\M_s^+$ and $\M_h^+$ in each color.
\item Prove that $\M_s^+$ and $\M_h^+$ satisfy the conclusion of Theorem \ref{underlying-cofibrant}.
\end{enumerate}
\end{problem}

Lastly, injective model structures on various categories of spectra have the property that all objects are cofibrant, so the condition about the forgetful functor preserving cofibrancy is trivial. However, not all operads are admissible. A likely remedy is to develop \textit{positive injective} model structures (by requiring cofibrations to be isomorphisms in level zero), which would automatically be Quillen equivalent to existing stable model structures on spectra, but the authors do not know a reference where this is done. 

\begin{problem}
Let $\M$ denote the category of symmetric spectra.
\begin{enumerate}
\item Prove that the positive injective stable model structure $\M_i^+$ is a monoidal model category.
\item Prove that all operads are admissible in $\M_i^+$. If so, then automatically cofibrant $\O$-algebras forget to cofibrant underlying objects.
\item Prove that $\M_i^+$ satisfies the conclusion of Theorem \ref{underlying-cofibrant}.
\item Do the same for symmetric spectra valued in a general base model category $\cat C$, where stabilization is with respect to an endofunctor $G$.
\item Do the same for orthogonal spectra and equivariant orthogonal spectra, possibly restricting to $\Delta$-generated spaces as is done in \cite{white-localization} (Section 8).
\item Produce a model structure on the category of $S$-modules, Quillen equivalent to the one in \cite{ekmm}, with the property that cofibrant commutative ring spectra are underlying cofibrant. Do the same for general entrywise cofibrant colored operads, and prove that the conclusion of Theorem \ref{underlying-cofibrant} holds in this setting.
\end{enumerate}
\end{problem}


\section{Semi-Model Categories and $\infty$-Categories for Operad Algebras} \label{sec:appendix}

In this paper, we often transferred model structures, using $(\spadesuit)$, or semi-model structures, using Def. \ref{heart-conditions} or using $\Sigmac$-cofibrant operads $\O$, to categories of $\O$-algebras. The language of $\infty$-categories could also be used to study the homotopy theory of $\O$-algebras. We work in the model of quasi-categories, i.e., everywhere we write $\infty$-category we mean quasi-category. The main results of this section, Theorems \ref{thm:appendix} and \ref{thm:appendix2}, show that the two approaches---namely, semi-model categories and $\infty$-categories---are equivalent in a suitable sense for $\Sigmac$-cofibrant $\fC$-colored operads that are \emph{not} necessarily admissible.

\subsection{Preliminaries on $\infty$-operads} \label{subsec:prelim-infty-operads}

As detailed in \cite{lurie-higher-algebra} (4.5.4.7 and 4.5.4.12), the crucial property needed to compare a model structure on $\O$-algebras with the corresponding $\infty$-category structure, is that the forgetful functor
\[U: \algom \to \M^\fC\]
preserves and reflects homotopy sifted colimits, that is, $N(\cat C)$-indexed homotopy colimits, where $\cat C$ is a small category such that the nerve $N(\cat C)$ is sifted \cite{lurie} (5.5.8.1).

Lurie \cite{lurie-higher-algebra} (4.5.4.12) proves this property for the $\Com$-operad and a restrictive class of model categories $\M$: namely, combinatorial and freely powered (4.5.4.2) monoidal model categories. Lurie then deduces (4.5.4.7) that the underlying $\infty$-category $N(\CAlg(\M)^c)[W_{Com}^{-1}]$ of the model category $\CAlg(\M)$--where $(-)^c$ refers to taking cofibrant objects, and $W_{Com}$ is the class of weak equivalences of $\Com$-algebras--is equivalent as an $\infty$-category to $\CAlg(N(\M^c)[W^{-1}])$, obtained as the $\infty$-category of commutative monoids valued in the symmetric monoidal $\infty$-category $N(\M^c)[W^{-1}]$ associated to $\M$. Here $N(\M^c)$ denotes the homotopy coherent nerve of the simplicial category $\M^c$, and the notation $(-)[W^{-1}]$ refers to the $\infty$-categorical meaning of inverting the class $W$ \cite{lurie-higher-algebra} (1.3.4.1). To be precise, the $\infty$-category $N(\M^c)[W^{-1}]$ can be constructed via a fibrant replacement of the pair $(\M^c,W)$ in the category of marked simplicial sets \cite{lurie-higher-algebra} (1.3.4.1).

Following the model of Lurie's proof, it is possible to prove that, whenever $\M$ is a simplicial monoidal model category and $\O$ is an \emph{admissible} $\Sigmac$-cofibrant simplicial colored operad (Def. \ref{spade} and \ref{def:sigma-cof}), then the forgetful functor preserves and reflects homotopy sifted colimits, and the $\infty$-category obtained from the model category of $\O$-algebras is equivalent as an $\infty$-category to the $\infty$-category obtained from $N^\otimes\O$-algebras in the $\infty$-category associated to $\M$ \cite{dmitri} (7.9 and 7.11). Here $N^\otimes \O$ is the operadic nerve of $\O$ \cite{lurie-higher-algebra} (2.1.1.23), i.e., Lurie's model for the $\infty$-operad associated to the simplicial colored operad $\O$.
Consequently, for \emph{admissible} $\Sigmac$-cofibrant colored simplicial operads, the homotopy theory obtained via the model category route matches the homotopy theory obtained via the $\infty$-category route. 

We extend this result in two ways. First, we will show that it holds when $\O$ is only \textit{semi-admissible} instead of admissible (i.e., $\algom$ has a transferred semi-model structure). Second, we will show the same thing for the setting of enriched $\infty$-operads. For the latter, we work in a monoidal model category $\M$ (not-necessarily simplicial) and consider a colored operad $\O$ valued in $\M$. Note that if $\M$ is a $\cat V$-model category for some monoidal model category $\cat V$, and $\O$ is a colored operad valued in $\cat V$, then there is a colored operad $\O'$ valued in $\M$ with the same algebras (obtained by tensoring the levels of $\O$ with the unit of $\M$), so we focus on the case when $\O$ is valued in $\M$. In this case, there is an associated \textit{enriched $\infty$-operad} \cite{chu-haugseng} as we now describe. First, we must restate \cite{haugseng-2019} (4.1).

\begin{definition} \label{defn:subcat-flat}
Let $\M$ be a monoidal model category. A \textit{subcategory of flat objects} is a full symmetric monoidal subcategory $\M^\flat$ (which implies the unit is flat) that satisfies the following two conditions:
\begin{enumerate}
\item All cofibrant objects are flat (that is, are in $\M^\flat$).
\item If $X$ is flat and $f$ is a weak equivalence in $\M^\flat$, then $X\otimes f$ is a weak equivalence.
\end{enumerate}
\end{definition}

If the unit of $\M$ is cofibrant, then the subcategory of cofibrant objects is a subcategory of flat objects \cite{haugseng-2019} (4.2), by Ken Brown's lemma. We note that, if the unit of $\M$ is cofibrant, then the same is true for both $\arrowmppproj$ and $\arrowmtensorinj$. The purpose of the definition above is to avoid assuming the monoidal unit is cofibrant, as this would rule out positive (flat) model structures on spectra (which do admit a subcategory of flat objects, namely the cofibrant objects of the flat model structure, by \cite{haugseng-2019} (4.11)). 
In \cite{white-commutative} and \cite{white-localization}, the first author gives many examples of model categories with a subcategory of flat objects (namely, the subcategory of cofibrant objects), including spaces, simplicial sets, chain complexes, diagram categories, simplicial presheaves, and various categories of spectra. 


With Definition \ref{defn:subcat-flat} in hand, we are ready to describe the enriched $\infty$-operad associated to a colored operad $\O$ valued in $\M$, following \cite{haugseng-2019} (Section 4). First, the inclusions $\M^c \hookrightarrow \M^\flat \hookrightarrow \M$ induce equivalences of localizations when all three are localized with respect to their subcategories of weak equivalences. Next, the symmetric monoidal localization $\M^\flat \to \M^\flat[W^{-1}]\simeq \M[W^{-1}]$ of \cite{lurie-higher-algebra} (4.1.7.4) gives a functor from $\infty$-operads enriched in $\M^\flat$ to $\infty$-operads enriched in $\M[W^{-1}]$. But because $\M^\flat$ is a 1-category, the former are simply strict colored operads in $\M^\flat$. The following is a combination of \cite{chu-haugseng} (1.1.3) and \cite{haugseng-2019} (4.4).

\begin{proposition}
Let $\M$ be a symmetric monoidal model category and $\M^\flat$ a subcategory of flat objects. Then the $\infty$-category of $\infty$-operads enriched in $\M[W^{-1}]$ is equivalent to the $\infty$-category of enriched colored operads in $\M^\flat$, with the Dwyer-Kan equivalences inverted. 
\end{proposition}


With these preliminary results and definitions in hand, we are ready to prove the main results of the section.

\subsection{Homotopy Sifted Colimits}

Following the model of \cite{lurie-higher-algebra} (4.5.4.7 and 4.5.4.12), we must first prove that the forgetful functor
\[U: \algom \to \M^\fC\]
preserves and reflects homotopy sifted colimits, even when $\algom$ is only a semi-model category. It suffices to prove this in the case where $\O$ is a colored operad in $\M$, as the case where $\O$ is a simplicial colored operad and $\M$ is a simplicial model category follows from our discussion above regarding $\cat V$-model categories. 

It is known that for every cofibrantly generated monoidal model category $\M$, every $\Sigmac$-cofibrant colored operad $\O$ in $\M$ is semi-admissible.  In other words, there is a transferred semi-model structure on $\O$-algebras \cite{white-yau} (6.3.1). An alternative approach assumes $\M$ satisfies $(\clubsuit)$ and appeals to \cite{white-yau} (6.2.3) for such a semi-model structure. It is also known that there are $\Sigmac$-cofibrant colored operads $\O$ whose category of $\O$-algebras do \emph{not} admit a full model structure \cite{batanin-white-eilenberg-moore} (2.9). Hence, the results in this section really do apply to previously unknown examples, and complete the study of semi-model structures on operad-algebras set out in \cite{white-yau, white-yau2, white-yau3, white-yau4}. For completeness, we handle the case of both symmetric and non-symmetric colored operads \cite{muro}, noting that for the non-symmetric case, being $\Sigmac$-cofibrant is the same as being entrywise cofibrant.


\begin{proposition} \label{prop:sifted}
Suppose $\M$ is a cofibrantly generated monoidal model category and $\O$ is a $\Sigmac$-cofibrant (symmetric) $\fC$-colored operad valued in $\M$. Then the forgetful functor $U: \algom \to \M^\fC$ preserves and reflects homotopy sifted colimits. 
\end{proposition}

\begin{proof}
We follow the proof from \cite{dmitri} (7.9), which is itself based on the proof of \cite{lurie-higher-algebra} (4.5.4.12). First, as pointed out in \cite{lurie-higher-algebra}, the reflection property is implied by the preservation property, and it is sufficient to prove that $U$ preserves homotopy colimits indexed by a small category $\cat D$ such that the nerve $N(\cat D)$ is homotopy sifted. 

Consider the projective model structure $(\M^\fC)^{\cat D}$, the projective semi-model structure $\algom^{\cat D}$ guaranteed by \cite{barwickSemi} (3.4), and the forgetful functor 
\[U^{\cat D}: \algom^{\cat D} \to (\M^\fC)^{\cat D}.\] 
Let 
\[F: (\M^\fC)^{\cat D} \to \M^\fC \andspace F_{\algo}:\algom^{\cat D} \to \algom\]
denote the colimit functors with respect to $\cat D$. The proof in \cite{lurie-higher-algebra} (4.5.4.12) reduces us to proving that the canonical isomorphism of functors 
\[\alpha: F\circ U^{\cat D} \cong U\circ F_{\algo}: \algom^{\cat D} \to \M^\fC\]
persists after everything is derived. 

Let $LF$ and $LF_\algo$ denote the left derived functors of $F$ and $F_\algo$, obtained via cofibrant replacement in $(\M^\fC)^{\cat D}$ and $\algom^{\cat D}$, respectively.  Since $U$ and $U^{\cat D}$ preserve weak equivalences, as in \cite{lurie-higher-algebra} (4.5.4.12), we are reduced to proving that the induced natural transformation $\overline{\alpha}: LF \circ U^{\cat D} \to U \circ LF_{\algo}$ is an isomorphism in the homotopy category. This means that, for every cofibrant $A$ in $\algom^{\cat D}$, we must show that 
\[\overline{\alpha}:  LF(U^{\cat D}A) \to U (LF_{\algo}(A))\]
is a weak equivalence. 

The right hand side is canonically weakly equivalent to $U(F_\algo(A))$ because $A$ is projectively cofibrant, and this is weakly equivalent to $F(U^{\cat D}A)$ via $\alpha$. At this point, the proof in \cite{lurie-higher-algebra} (4.5.4.12) requires a detailed analysis of so-called ``good'' objects and morphisms in $(\M^\fC)^{\cat D}$. However, when $\O$ is $\Sigmac$-cofibrant, the situation is much simpler, because $U$ takes cofibrant algebras to cofibrant objects of $\M^\fC$ \cite{white-yau} (6.3.1) (\cite{muro} (9.5) for the non-symmetric case). 

Furthermore, the $\cat D$-constant operad $\O^{\cat D}$, taking value $\O$ at every $a\in \cat D$, is $\Sigmac$-cofibrant in $\algom^{\cat D}$.  This can be seen directly, as $\Sigmac$-cofibrancy for an operad $P$ valued in $\M^{\cat D}$ is the condition that, for each $a\in \cat D$ and each $(\uc; d) \in \Sigmacopc$, the object $P_a \duc$ ($=\O \duc$ in our case) is projectively cofibrant in $\M^{\Sigmacopc}$.  Hence, by \cite{white-yau} (6.3.1) (\cite{muro} (9.5) for the non-symmetric case), the functor $U^{\cat D}$ also preserves cofibrancy, since the projective semi-model structure transferred from the semi-model structure on $\algom$ is the same as the transferred semi-model structure on $\O^{\cat D}$-algebras in the projective model structure $(\M^\fC)^{\cat D}$. Hence, $U^{\cat D}A$ is cofibrant in $(\M^\fC)^{\cat D}$, and so $LF(U^{\cat D}A) \simeq F(U^{\cat D}A)$ as required. 
\end{proof} 

\begin{remark}
Following the model of \cite{lurie-higher-algebra} (or \cite{dmitri}), after establishing Proposition \ref{prop:sifted}, the next step should be to prove that the semi-model category $\algom$ describes the $\infty$-category of $N^\otimes \O$-algebras in the $\infty$-category associated to $\M$, as discussed above. However, when $\algom$ is only a semi-model structure, an additional step is needed. We need to know that homotopy colimits (given by colimits of projectively cofibrant objects in $\algom^{\cat D}$) agree with $\infty$-categorical colimits. In the case of full model structures, one knows that the projective model structure on $\algom^{\cat D}$ describes the $\infty$-category of functors, and that a Quillen adjunction gives rise to an adjunction of $\infty$-categories. For the case of semi-model categories, we invoke \cite{giulio} (A.10) for the latter.
\end{remark}

\begin{remark}
We conjecture that Proposition \ref{prop:sifted} remains true for entrywise cofibrant colored operads $\O$, if $\M$ satisfies $(\clubsuit)$, and if we replace appeals to \cite{white-yau} (6.3.1) above by appeals to \cite{white-yau} (6.2.3). However, the proof of this would require a detailed analysis of `good' objects and would take us too far afield.
\end{remark}

\subsection{Semi-Model Categories and $\infty$-Categories of Operad Algebras}

With the previous proposition in hand, we are ready for the main result of this section. The slogan for this result is that, for any $\Sigmac$-free (symmetric) colored operad $\O$ and any reasonable monoidal model category $\M$, the semi-model category of $\O$-algebras in $\M$ describes the corresponding $\infty$-category of $\O$-algebras in the symmetric monoidal $\infty$-category described by $\M$. This is true both of the following cases:

\begin{enumerate}
\item The unenriched case: where $\M$ is a simplicial monoidal model category, $\O$ is a simplicial colored operad, and the $\infty$-operad associated to $\O$ is the operadic nerve $N^{\otimes}\O$ of $\O$ \cite{lurie-higher-algebra} (2.1.1.23). 
\item The enriched case: where $\M$ is a monoidal model category, $\O$ is a colored operad valued in $\M$, and we use the theory of enriched $\infty$-operads to define the $\infty$-category of $\O$-algebras (as recalled in Section \ref{subsec:prelim-infty-operads} and spelled out in \cite{chu-haugseng, haugseng-2019}).
\end{enumerate}

For both cases, we handle the cases where $\O$ is a symmetric colored operad and where $\O$ is a non-symmetric colored operad simultaneously. We handle the enriched case first.

\begin{theorem}  \label{thm:appendix}
Suppose $\M$ is a cofibrantly generated monoidal model category that admits a subcategory of flat objects $\M^\flat$, and $\O$ is a $\Sigmac$-cofibrant (symmetric) $\fC$-colored operad valued in $\M^\flat$. 
\begin{itemize}
\item Denote by $\algom^c[W_\O^{-1}]$ the $\infty$-category obtained from the semi-model category $\algom$, by first passing to the subcategory of cofibrant objects, and then inverting the weak equivalences between $\O$-algebras.
\item Denote by $\algominv$ the $\infty$-category obtained by first passing from $\M$ to the (symmetric) monoidal category $\M[W^{-1}]$ and then passing to $\O$-algebras, where $\O$ is viewed as a colored operad in $\M[W^{-1}] \simeq \M^\flat[W^{-1}]$.
\end{itemize}
Then the  natural comparison functor
\[\algom^c[W_\O^{-1}] \to \algominv\]
is an equivalence of $\infty$-categories.
\end{theorem} 

\begin{proof}
The proof of \cite{haugseng-2019} (4.10) goes through directly by replacing the appeal to \cite{dmitri} (7.8) with an appeal to Proposition \ref{prop:sifted}. That is, we consider the forgetful functors from both categories to the $\infty$-category associated to $\M^\fC$, and appeal to the Barr-Beck theorem for $\infty$-categories \cite{lurie-higher-algebra} (4.7.3.16) to see that these forgetful functors are monadic right adjoints (this is where Proposition \ref{prop:sifted} is needed). We appeal to \cite{haugseng-2019} (3.8), which occurs entirely on the $\infty$-category level, for the usual formula for free $\O$-algebras and the observation that the two associated monads on $\M^\fC$ have equivalent underlying endofunctors. This proof works for both symmetric and non-symmetric colored operads $\O$, as both are known to inherit transferred semi-model structures from $\M^\fC$, and as Proposition \ref{prop:sifted} applies in both settings.
\end{proof}

\begin{remark} \label{remark:mazel-gee}
The proof of \cite{haugseng-2019} (4.10) relies on the observation that a Quillen adjunction $F:\M \rightleftarrows \N:G$ induces an adjunction between the underlying $\infty$-categories. We appeal to \cite{giulio} (A.10) for the semi-model category analogue of this fact.

\end{remark}

We turn now to the unenriched case.

\begin{theorem}  \label{thm:appendix2}
Suppose $\M$ is a cofibrantly generated simplicial monoidal model category, and $\O$ is a $\Sigmac$-cofibrant (symmetric) simplicial $\fC$-colored operad. 
\begin{itemize}
\item Denote by $N(\algom^c)[W_\algo^{-1}]$ the $\infty$-category obtained from the semi-model category $\algom$, by first passing to the subcategory of cofibrant objects, then taking the nerve, and then inverting the weak equivalences.
\item Denote by $\algnom$ the $\infty$-category of $N^{\otimes}\O$-algebras valued in the $\infty$-category $N(\M^c)[W^{-1}]$ associated to $\M$.
\end{itemize}
Then the  natural comparison functor
\[N(\algom^c)[W_\algo^{-1}] \to \algnom\]
is an equivalence of $\infty$-categories.
\end{theorem}

\begin{proof}
We deliberately phrased the proof of Theorem \ref{thm:appendix}, so that word-for-word it proves this result as well (again with the critical step hinging on an appeal to Proposition \ref{prop:sifted}). We only stated the two theorems separately to highlight the difference between enriched and unenriched $\infty$-operads, and the connection to where the colored operad $\O$ is valued.
\end{proof}

\begin{remark}
One can show that Theorems \ref{thm:appendix} and  \ref{thm:appendix2} are false in general in the symmetric case if the $\Sigmac$-cofibrancy of $\O$ is dropped. Well-known counterexamples include the operad $\Com$ and $\M = \Ch(\F_p)$. 
However, every $\fC$-colored operad $\O$ admits a $\Sigmac$-cofibrant replacement $Q\O$. If $\O$ is semi-admissible and admits rectification with $Q\O$ (meaning, there is a Quillen equivalence of semi-model categories between $\algom$ and $\alg(Q\O;\M)$), then Theorems \ref{thm:appendix} and  \ref{thm:appendix2} do apply to $\O$, since the weak equivalence $Q\O \to \O$ induces an equivalence $N^\otimes Q\O \to N^\otimes \O$, and hence we can use the two out of three property to deduce the statement for $\O$ from the statement for $Q\O$. Conditions on $\M$ under which rectification hold are provided in \cite{white-commutative} (for $Com$ rectifying to $E_\infty$) and \cite{white-yau3} (for general colored operads) among other places.
\end{remark}

\begin{remark}
Theorem \ref{thm:appendix} answers positively the question raised in \cite{haugseng-2019} (4.13) about extending \cite{haugseng-2019} (4.10) to $\Sigma$-cofibrant operads and semi-model category structure on $\algom$. As pointed out by Haugseng, the assumptions on $\M$ and $\O$ are much weaker than those required to get a full model structure on $\O$-algebras. In particular, Theorem \ref{thm:appendix} applies not only to the examples listed by Haugseng---namely, spaces, simplicial sets, chain complexes, and symmetric spectra---but also to equivariant spaces, equivariant orthogonal spectra, motivic symmetric spectra, the stable module category, chain complexes over a field of nonzero characteristic, simplicial presheaves, the projective model structure on small functors \cite{chorny-white}, the folk model structure on the category of small categories (or groupoids), to various abelian model structures arising from the theory of cotorsion pairs, and to left Bousfield localizations of these categories. 

These examples are detailed in our papers \cite{white-commutative, white-yau, white-yau2, white-localization}. In several of these examples (e.g., chain complexes over a field of nonzero characteristic, examples arising from cotorsion pairs, and algebras over left Bousfield localizations $L_{\cat C}\M$), categories of algebras are known to have transferred semi-model structures but are not known to have transferred model structures. For chain complexes over $\mathbb{F}_2$, there is even an explicit example of a category of $\O$-algebras that has a transferred semi-model structure that is not a model structure \cite{batanin-white-eilenberg-moore} (2.9). For algebras over a left Bousfield localization $L_{\cat C}\M$, many examples are discussed in \cite{crm, bous-loc-semi, white-oberwolfach, Reedy-paper}.

In most of the examples listed above, the unit is cofibrant and cofibrant objects are flat, so the category of cofibrant objects is our $\M^\flat$ (note that left Bousfield localization does not change the class of cofibrant objects). For the positive (flat) model structure on equivariant orthogonal spectra (resp., motivic symmetric spectra), one can use the cofibrant objects of the flat model structure, just as Haugseng does for symmetric spectra \cite{haugseng-2019}, as discussed in \cite{hovey-white} (resp., \cite{dmitri}, building on work of Hornbostel). 
\end{remark}

We conclude with a specialization of Theorem \ref{thm:appendix} to the main examples of interest in the present paper.

\begin{lemma} \label{lemma:flat-objects}
Suppose $\M$ is a monoidal model category that admits a subcategory of flat objects, $\M^\flat$. Then $\arrowmtensorinj$ also admits a subcategory of flat objects.
\end{lemma}

\begin{proof}
In $\arrowmtensorinj$, we take the full subcategory consisting of arrows $f: X_1\to X_2$ where $X_1$ and $X_2$ are in $\M^\flat$. This is a symmetric monoidal subcategory of $\arrowmtensorinj$, as the monoidal unit $\Id : \tensorunit \to \tensorunit$ is flat and the tensor product of two flat arrows is flat. Condition (1) of Definition \ref{defn:subcat-flat} holds because cofibrations are entrywise, and (2) holds because the tensor product and weak equivalences are entrywise.
\end{proof}

\begin{corollary}
Suppose $\M$ is a cofibrantly generated monoidal model category that admits a subcategory of flat objects $\M^\flat$. Suppose $\O$ is a $\Sigmac$-cofibrant $\fC$-colored operad valued in $\M^\flat$. Then the transferred semi-model structures of Corollary \ref{sigmacof-smith=map} on $\algomtensor$ and $\algompp$ describe the corresponding $\infty$-categories, in the sense of Theorem \ref{thm:appendix}. If, in addition, $\M$ is stable, then the Quillen equivalence of Corollary \ref{sigmacof-smith=map} yields an equivalence of $\infty$-categories.
\end{corollary}

\begin{proof}
This follows from Theorem \ref{thm:appendix}, applied to:
\begin{itemize}
\item $\arrowmtensorinj$ and the colored operad $\arrowotensor$, appealing to Lemma \ref{lemma:flat-objects} for the subcategory of flat objects and to Proposition \ref{sigmac-cofibrant-arrrowcat} for the $\Sigmac$-cofibrancy, and
\item $\M$ and the colored operad $\O^s$, with the assumed subcategory of flat objects on $\M$. As Corollary \ref{smitho-m} shows, $\O^s$ is $\Sigma_{\fC \sqcup \fC}$-cofibrant, and the transferred semi-model structure on $\O^s$-algebras coincides with the transferred semi-model structure on $\algompp$.
\end{itemize}
The statement about Quillen equivalences follows from \cite{giulio} (A.11).
\end{proof}

We note that, in the examples mentioned after Definition \ref{defn:subcat-flat}, we could take $\M^\flat$ to be the subcategory of cofibrant objects of $\M$. In these examples, every $\Sigmac$-cofibrant $\fC$-colored operad is already entrywise cofibrant. Hence, it is no loss of generality to assume $\O$ is valued in $\M^\flat$ instead of in $\M$, for these examples.


\end{document}